\documentclass{article}
\usepackage{float}

\usepackage{arxiv}

\usepackage[utf8]{inputenc}
\usepackage[T1]{fontenc}
\usepackage{hyperref}
\usepackage{url}
\usepackage{booktabs}
\usepackage{amsfonts,amssymb}
\usepackage{nicefrac}
\usepackage{microtype}
\usepackage{graphicx}
\usepackage{natbib}
\usepackage{doi}
\usepackage{babel}
\usepackage{amsmath,amsthm,bm}
\usepackage{tikz}
\usepackage{xcolor}
\usepackage{pgfplots}
\pgfplotsset{compat=newest}
\usepackage{enumerate}

\newcommand{\bu}{\boldsymbol{u}}
\newcommand{\bx}{\boldsymbol{x}}
\newcommand{\bxs}{\boldsymbol{x}_{s}}
\newcommand{\bxf}{\boldsymbol{x}_{f}}
\newcommand{\bvf}{\boldsymbol{v}_{f}}

\newtheorem{remark}{Remark}
\newtheorem{theorem}{Theorem}[section]

\title{A variational model of nonlinear poroelasticity}

\author{ \href{https://orcid.org/0000-0002-6603-8840}{\includegraphics[scale=0.06]{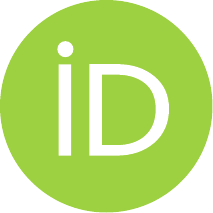}\hspace{1mm}James H.~Adler} \\
	Department of Mathematics\\
	Tufts University\\
	Medford, MA 02155 \\
	\texttt{james.adler@tufts.edu} \\
	\And
	\href{https://orcid.org/0000-0001-7533-0416}{\includegraphics[scale=0.06]{orcid.pdf}\hspace{1mm}Xiaozhe Hu} \\
	Department of Mathematics\\
	Tufts University\\
	Medford, MA 02155 \\
	\texttt{xiaozhe.hu@tufts.edu} \\
	\And
	\href{https://orcid.org/0000-0002-1205-3359}{\includegraphics[scale=0.06]{orcid.pdf}\hspace{1mm}Arkadz Kirshtein}\\
	Department of Mathematics \& Statistics\\
	Texas A\&M University -- Corpus Christi\\
	Corpus Christi, TX 78412 \\
	\texttt{arkadz.kirshtein@tamucc.edu} \\
}

\renewcommand{\shorttitle}{A variational model of nonlinear poroelasticity}

\hypersetup{
pdftitle={A variational model of nonlinear poroelasticity},
pdfsubject={},
pdfauthor={James H. Adler, Xiaozhe Hu, Arkadz Kirshtein},
pdfkeywords={poroelasticity, variational modeling, energy stability, mixed finite elements, nonlinear Biot},
}

\begin{document}
\maketitle

\begin{abstract}
We derive a thermodynamically-consistent model of fluid flow through a
poroelastic medium. Starting from elastic and fluid free-energy densities,
an energy-dissipation rate, and a kinematic constraint, the force-balance
equations are derived using variational principles, with the pressure--density
constitutive relation emerging as a direct consequence of the variational
structure; the same kinematic constraint also supplies the total-flux
transport structure.
In the ideal-gas limit, the model linearization recovers the classical linear Biot equations.
For power-law fluid energies, it yields isentropic pressure--density relations.
A key advantage of the variational formulation is that extensions to richer
physics, such as thermal effects, chemical reactions, or multi-component fluids,
can be incorporated systematically by augmenting the energy and dissipation
functionals without redesigning the force-balance or transport closure.
We support the model with an energy-compatible two-field discretization and study
consolidation under a surface load with three lateral-boundary treatments
and three fluid-compressibility exponents.
\end{abstract}


\section{Introduction}
Poroelasticity describes coupled deformation and fluid flow in porous media, with foundational applications in geomechanics, biomechanics, and subsurface transport \citep{biot1941general,biot1955theory,coussy2004poromechanics,lewis1998finite}. In particular, classical Biot theory linearizes around a reference state, yielding a well-understood system with constant permeability and a linear storage coefficient. In the Biot model, the motion of fluid in a porous medium and the deformation of the porous medium are governed by Darcy's law and the linear elasticity equation, respectively. It dates back to the one-dimensional work of Terzaghi \citep{terzaghi1943theoretical}, with the three-dimensional model later developed by Biot \citep{biot1941general,biot1955theory}.

The classical Biot equations have been derived using several different approaches. First, Biot's own derivation was based on linear constitutive relations among stress, pore pressure, and fluid content, calibrated against measurable moduli \citep{biot1941general,biot1955theory,biot1962mechanics}. Second, mixture theory derives the macroscopic balance laws by superimposing solid and fluid continua, each governed by its own balance laws \citep{bowen1980incompressible}. \citet{coussy1998frommixture} showed explicitly how the macroscale field equations obtained this way can be recast in terms of the measurable quantities used in Biot's approach, and \citet{steeb2019mechanics} give a derivation of linear poroelasticity from the continuum mixture theory perspective. Third, averaging theory formally averages the pore-scale equations of motion and stress--strain relations over a representative volume; \citet{pride1992deriving} carried this out explicitly for a two-phase fluid/solid isotropic medium and showed the result to be consistent with Biot's equations of motion and stress-strain relations. Finally, the fourth approach, homogenization via the two-scale method, upscales pore-scale elasticity and Stokes flow with appropriate conditions at the solid-fluid boundary, recovering Biot's equations when the dimensionless viscosity of the fluid is small \citep{burridge1981poroelasticity,auriault2009homogenization}.

However, in many settings of practical interest, the fluid is compressible in a nonlinear way, the skeleton undergoes finite deformation, or the poroelastic system is coupled to additional physical processes such as thermal effects, chemical reactions, or multi-component transport. Finite-deformation, fully Eulerian extensions of Biot theory exist \citep{chapelle2014,rohan2017}, but at the cost of substantially more elaborate modeling and discretization. The present work instead pursues the first of these directions, nonlinear fluid compressibility, while keeping the solid in the same small-strain regime as the classical theory. Extending the Biot framework in a thermodynamically consistent manner requires a modeling approach that treats constitutive closure, force balance, and dissipation in a unified way.

The variational energy-dissipation framework \citep{giga2018variational} provides exactly this structure. Starting from energy and dissipation functionals, variation with respect to the solid and fluid degrees of freedom yields the force-balance equations, with thermodynamic consistency guaranteed by construction. This point of view is well established in continuum mechanics and geometric mechanics \citep{marsden1978foundations,gurtin1981continuum} and has been used successfully to derive models for complex fluids, phase-field systems, and other dissipative multiphysics problems.

In this paper, we apply the energy-dissipation framework to nonlinear
poroelasticity. Taking a free-energy density that combines elastic and fluid
contributions, and a dissipation functional for viscous Darcy resistance,
the variational procedure yields a coupled system in the elastic displacement
$\bu$ and fluid density $\rho_f$. The constitutive pressure relation
\[
p(\rho_f)=\rho_f\,\omega_{\rho}(\rho_f)-\omega(\rho_f)
\]
emerges as a direct consequence of the variation; the total-flux transport structure $\boldsymbol{j}=\alpha\rho_f\partial_t\bu + \boldsymbol{q}_{\mathrm{Darcy}}$ follows from the kinematic coupling assumption. In the ideal-gas limit, $\omega=M\rho_f\ln\rho_f$, the model linearization recovers the classical linear Biot equations exactly.  For power-law energies, $\omega=M\rho_f^\gamma$, it gives the isentropic pressure law $p = M(\gamma-1)\rho_f^\gamma$. The solid stays in this small-strain regime throughout, and linear elasticity is sufficient for $\sigma_e$ once a Lagrangian description, described in Section~\ref{sec:model-derivation}, is adopted.  Therefore, the nonlinearity studied here is specifically the fluid's compressibility law, not finite-strain solid mechanics.

Just as the elastic strain energy determines the solid's stress response, the present derivation assigns the fluid its own free-energy density $\omega(\rho_f)$ on the same footing, from which the pressure-density relation follows by variation. Linearizing this energy in the ideal-gas case recovers exactly the Biot modulus $M$ and coupling coefficient $\alpha$ of the classical theory \citep{biot1941general,biot1955theory}. A central virtue of tracking explicitly which relations follow from varying this combined energy, and which follow from the kinematic coupling between the phases, is extensibility. Coupling to thermal effects \citep{deanna2019nonisothermal,liu2018nonisothermal}, reactive transport \citep{wang2020field}, or multi-component fluids \citep{brannick2016dynamics} requires only adding the appropriate free-energy and dissipation terms to the functionals. The force-balance equations are then re-derived by variation, and the same kinematic coupling assumption supplies the corresponding transport closure, guaranteeing that the extended model inherits the same thermodynamic structure. This systematic coupling pathway is a principal motivation for the present work.

The remainder of the paper is organized as follows. Section~\ref{sec:model-derivation} presents the variational derivation and boundary-condition setting. In Section~\ref{sec:twofield-discretization}, we support the model with a compatible two-field discretization in $(\bu,\rho_f)$ that inherits a discrete energy identity. Section~\ref{sec:numerical-experiments} presents numerical experiments that study the model's consolidation behavior across three lateral-boundary treatments (fixed, roller, and free) and three fluid-compressibility exponents $\gamma \in \{1, 2, 5\}$. Finally, Section~\ref{sec:conclusions} gives some concluding remarks and directions for future work.

\section{Model derivation}
\label{sec:model-derivation}

The derivation in this section follows a variational modeling strategy for complex fluids and continua \citep{giga2018variational,marsden1978foundations,gurtin1981continuum}.
Here, we consider the background elastic matrix moving according to
flow map $\boldsymbol{x_{s}}\left(\bx,\,t\right)$ and
fluid, whose motion given by velocity $\bvf$ is described
relative to the elastic matrix. Thus we assume the following constitutive
kinematic relation for the density of the fluid $\rho_{f}:$
\begin{equation}
\partial_{t}\rho_{f}+\nabla\cdot\left(\left(\alpha\partial_{t}\bxs+\bvf\right)\rho_{f}\right)=0.\label{eq:rhof-xs}
\end{equation}
 Constant $\alpha\in\left[0,\,1\right]$ is the coefficient that controls
the slip between elastic matrix and the fluid, where $\alpha=1$ corresponds
to a no-slip case with fluid being carried with the solid. In the small
elastic deformation approximation, we neglect the difference between
Lagrangian coordinates for elastic deformation and Eulerian coordinates.
Consider the following energy dissipation law:
\[
\frac{d}{dt}\int W_{e}\left(\nabla\bxs\right)+\omega\left(\rho_{f}\right)d\bx=-\int\frac{\rho_{f}}{\kappa}\left|\bvf\right|^{2}d\bx,
\]
where $W_{e}$ is the elastic internal energy and $\omega$ is the fluid internal
energy. Defining a fluid flow map $\bxf\left(\xi,\,t\right)$
satisfying $\partial_{t}\bxf\left(\xi,\,t\right)=\bvf\left(\bxf\left(\xi,\,t\right),\,t\right)$,
then using a virtual work principle for the fluid, we assume the variational
relation
\[
\delta\rho_{f}=-\nabla\cdot\left(\left(\alpha\delta\bxs+\delta\bxf\right)\rho_{f}\right),
\]
 and perform the variation of the free energy as follows:
\begin{align}
\delta\mathcal{F}= & ~\delta\int W_{e}\left(\nabla\bxs\right)+\omega\left(\rho_{f}\right)d\bx=\int\frac{\partial W_{e}\left(\nabla\bxs\right)}{\partial\nabla\bxs}:\nabla\delta\bxs+\omega_{\rho}\left(\rho_{f}\right)\delta\rho_{f}d\bx\nonumber \\
= & \int\frac{\partial W_{e}\left(\nabla\bxs\right)}{\partial\nabla\bxs}:\nabla\delta\bxs-\omega_{\rho}\left(\rho_{f}\right)\nabla\cdot\left(\left(\alpha\delta\bxs+\delta\bxf\right)\rho_{f}\right)d\bx\label{eq:variation}\\
= & \int-\left(\nabla\cdot\frac{\partial W_{e}\left(\nabla\bxs\right)}{\partial\nabla\bxs}\right)\cdot\delta\bxs+\left(\rho_{f}\nabla\omega_{\rho}\left(\rho_{f}\right)\right)\cdot\left(\alpha\delta\bxs+\delta\bxf\right)d\bx.\nonumber
\end{align}
 Thus, we write the variational derivatives as
\begin{align*}
\frac{\delta\mathcal{F}}{\delta\bxs}= & -\left(\nabla\cdot\frac{\partial W_{e}\left(\nabla\bxs\right)}{\partial\nabla\bxs}\right)+\alpha\rho_{f}\nabla\omega_{\rho}\left(\rho_{f}\right),\\
\frac{\delta\mathcal{F}}{\delta\bxf}= & \rho_{f}\nabla\omega_{\rho}\left(\rho_{f}\right).
\end{align*}
 Taking variation of the dissipation,
\[
\frac{\delta\mathcal{D}}{\delta\bvf}=\frac{\delta\frac{1}{2}\int\frac{\rho_{f}}{\kappa}\left|\bvf\right|^{2}d\bx}{\delta\bvf}=\frac{\rho_{f}}{\kappa}\bvf,
\]
 and writing the two-component force balance, $\left(\frac{\delta\mathcal{F}}{\delta\bxs},\,\frac{\delta\mathcal{F}}{\delta\bxf}+\frac{\delta\mathcal{D}}{\delta\bvf}\right)=\left(0,\,0\right)$,
together with the constitutive relation \eqref{eq:rhof-xs}, we obtain
the system with 3 unknowns $\left(\bxs,\,\rho_{f},\,\bvf\right)$:
\[
\begin{cases}
-\left(\nabla\cdot\frac{\partial W_{e}\left(\nabla\bxs\right)}{\partial\nabla\bxs}\right)+\alpha\rho_{f}\nabla\omega_{\rho}\left(\rho_{f}\right)=0,\\
\rho_{f}\nabla\omega_{\rho}\left(\rho_{f}\right)+\frac{\rho_{f}}{\kappa}\bvf=0,\\
\partial_{t}\rho_{f}+\nabla\cdot\left(\left(\alpha\partial_{t}\bxs+\bvf\right)\rho_{f}\right)=0.
\end{cases}
\]
Defining the displacement $\bu=\boldsymbol{x_{s}}\left(\bx,\,t\right)-\bx$,
pressure $p=\rho_{f}\omega_{\rho}\left(\rho_{f}\right)-\omega\left(\rho_{f}\right)$, flux $\boldsymbol{q}=\rho_{f}\bvf$, along with the
elastic stress
\[
\sigma_{e}\left(\bu\right)=\frac{\partial W_{e}}{\partial\nabla\bxs}\left(I+\nabla\bu\right),
\]
 we observe that
\[
\nabla p=\rho_{f}\nabla\omega_{\rho}\left(\rho_{f}\right).
\]
These definitions allow us to rewrite the system in terms of unknowns $\left(\bu,\,p,\,\boldsymbol{q}\right)$ as follows:
\begin{equation}
\begin{cases}
-\nabla\cdot\sigma_{e}\left(\bu\right)+\alpha\nabla p=0,\\
\boldsymbol{q}=-\kappa\nabla p,\quad p=\rho_{f}\omega_{\rho}\left(\rho_{f}\right)-\omega\left(\rho_{f}\right)\\
\partial_{t}\rho_{f}+\nabla\cdot\left(\alpha\rho_{f}\partial_{t}\bu\right)+\nabla\cdot\boldsymbol{q}=0.
\end{cases}\label{main-system-eq:}
\end{equation}

Additionally, assuming an internal energy of an ideal gas, $\omega\left(\rho\right)=M\rho\ln\rho$, we
see that $\frac{1}{M}p=\rho_{f}$. Thus, the system becomes
\[
\begin{cases}
-\nabla\cdot\sigma_{e}\left(\bu\right)+\alpha\nabla p=0,\\
\boldsymbol{q}=-\kappa\nabla p,\\
\frac{1}{M}\partial_{t}p+\nabla\cdot\left(\frac{\alpha}{M}p\partial_{t}\bu\right)+\nabla\cdot\boldsymbol{q}=0.
\end{cases}
\]

\subsection{Boundary conditions}

Let the boundary be partitioned as
\[
\partial\Omega=\Gamma_{c}\cup\Gamma_{f}\cup\Gamma_{s},\qquad
\Gamma_{i}\cap\Gamma_{j}=\emptyset\ \text{for}\ i\neq j,
\]
where $\Gamma_{c}$ is clamped, $\Gamma_{f}$ is traction-free, and $\Gamma_{s}$ carries a prescribed traction.
For transport, we consider impermeable walls,
\[
\boldsymbol{j}\cdot\boldsymbol{n}=0\quad\text{on }\partial\Omega,
\]
with total flux $\boldsymbol{j}=\boldsymbol{q}+\alpha\rho_{f}\partial_{t}\bu$.
From integration by parts in the variational derivation, one obtains the natural traction involving
$\omega_{\rho}(\rho_{f})\rho_{f}$. For physical modeling of loading, we instead impose traction through
the total pressure
\[
p=\rho_{f}\omega_{\rho}(\rho_{f})-\omega(\rho_{f}),
\]
which gives
\begin{align*}
\bu &=0, &&\text{on }\Gamma_{c},\\
\left[-\sigma_{e}(\bu)+\alpha pI\right]\cdot\boldsymbol{n} &=0, &&\text{on }\Gamma_{f},\\
\left[-\sigma_{e}(\bu)+\alpha pI\right]\cdot\boldsymbol{n} &=\boldsymbol{g}, &&\text{on }\Gamma_{s}.
\end{align*}
The same partition is used in the numerical discretizations below.

\begin{remark}[Biot limit]
If we consider a linearization (assuming linear elastic stress $\sigma_{e}$)
of the proposed system near the state $\left(\bu,\,p,\,\boldsymbol{q}\right)\approx\left(0,\,M,\,0\right)$ with constant $\alpha$, and note that under the ideal-gas closure, $p=M\rho_f$, this reference pressure
corresponds to a reference density $\rho_f=1$, then we recover the standard Biot system
\[
\begin{cases}
-\nabla\cdot\sigma_{e}\left(\bu\right)+\alpha\nabla p=0,\\
\boldsymbol{q}=-\kappa\nabla p,\\
\frac{1}{M}\partial_{t}p+\alpha\nabla\cdot\left(\partial_{t}\bu\right)+\nabla\cdot\boldsymbol{q}=0,
\end{cases}
\]
which is the classical consolidation form \citep{biot1941general,biot1955theory,coussy2004poromechanics}.
 In addition, when considering incompressibility of the fluid, we arrive
at the constraint on total fluid velocity
\[
\nabla\cdot\left(\alpha\partial_{t}\bu+\frac{\boldsymbol{q}}{\rho_{f}}\right)=0,
\]
which, with assumption $\rho_{f}\left(x,\,0\right)=1=\mathrm{const}$,
also leads to the standard equation in the incompressible Biot system,
\[
\nabla\cdot\left(\alpha\partial_{t}\bu\right)+\nabla\cdot\boldsymbol{q}=0.
\]
\end{remark}

\begin{remark}[power-law fluids]
Considering a fluidic internal energy in form of a power law,
\[
\omega\left(\rho\right)=M\rho^{\gamma},\,\gamma>1,
\]
 then one would obtain an isentropic pressure-density relation,
\[
p=M\left(\gamma-1\right)\rho^{\gamma}.
\]
\end{remark}

\section{Two-field discretization}
\label{sec:twofield-discretization}
To validate the model derived above, we consider a discretization of the system that is compatible with energy and variational principles.  To simplify the discrete model, 
the Darcy law $\boldsymbol{q}=-\kappa\nabla p(\rho_f)$ can be substituted
directly into the mass equation of~\eqref{main-system-eq:}, eliminating
$\boldsymbol{q}$ as a primary unknown.  This yields the two-field system
\begin{equation}
\begin{cases}
-\nabla\cdot\sigma_e(\bu)+\alpha\nabla p(\rho_f)=0,\\[4pt]
\partial_t\rho_f
+\nabla\cdot\!\bigl(\alpha\rho_f\,\partial_t\bu\bigr)
-\kappa\,\nabla\cdot\!\bigl(\nabla p(\rho_f)\bigr)=0,
\end{cases}
\label{eq:2f-strong}
\end{equation}
in the two unknowns $(\bu,\rho_f)$, closed by
$p(\rho_f)=\rho_f\omega_\rho(\rho_f)-\omega(\rho_f)$.
The Darcy flux, $\boldsymbol{q}=-\kappa\nabla p(\rho_f)$, is recovered a
posteriori once $\rho_f$ is known.

The weak variational model is then obtained by multiplying the momentum and mass equations by $\psi_{\bu}$ and $\psi_\rho$, respectively, and then integrating by parts.  Applying 
traction boundary conditions on the momentum equations and the total no-flux condition, $\boldsymbol{j}\cdot\boldsymbol{n}=0$
on all sealed boundaries, yields the weak form, posed over the standard Sobolev
spaces~\citep{evans2010partial},
\begin{equation}
\bigl(\sigma_e(\bu),\nabla\psi_{\bu}\bigr)
+\alpha\bigl(\nabla p(\rho_f),\psi_{\bu}\bigr)
=\bigl(\boldsymbol{g},\psi_{\bu}\bigr)_{\Gamma_s},\quad\forall \psi_{\bu}\in\mathcal{V}_{\bu}\subset \mathbf{H}^1,
\label{eq:2f-weak-u}
\end{equation}
\begin{equation}
\bigl(\partial_t\rho_f,\psi_\rho\bigr)
-\alpha\bigl(\rho_f\,\partial_t\bu,\nabla\psi_\rho\bigr)
+\kappa\bigl(\nabla p(\rho_f),\nabla\psi_\rho\bigr)
=0\quad \forall \psi_{\rho}\in\mathcal{V}_{\rho}\subset H^1.
\label{eq:2f-weak-rho}
\end{equation}
The total no-flux condition is not imposed as an essential (Dirichlet) constraint: it
arises as the natural boundary condition of the weak form~\eqref{eq:2f-weak-rho} and is
automatically satisfied on any boundary where no density Dirichlet condition is applied.
We discretize~\eqref{eq:2f-weak-u}--\eqref{eq:2f-weak-rho} in space
using the Galerkin finite-element method, choosing finite-dimensional subspaces, $\mathcal{V}^h_{\rho}$ and $\mathcal{V}^h_{\bu}$, for $H^1$ and its vector counterpart, $\mathbf{H}^1$.  In particular, we choose piecewise quadratic Lagrange finite elements, $\bm{\mathcal{P}}_2$ for $\mathcal{V}_{\bu}$ and piecewise linear finite elements, $\mathcal{P}_1$ for $\rho_f$.

\subsection{Linearized time discretization}\label{subsec:2f-scheme}

We discretize the system in time using a uniform time step $\delta t$ and set
$t^n=n\,\delta t$ for $n=0,1,\dots,N$; a superscript $n$ denotes the finite-element
approximation of the corresponding field at time $t^n$.
The continuous system~\eqref{eq:2f-strong} is nonlinear in two distinct ways: through
the constitutive pressure $p(\rho_f)$, and through the displacement--density coupling term
$\alpha\rho_f\partial_t\bu$, which is itself a product of the two unknowns. The scheme
constructed below linearizes the first of these exactly, by evaluating the constitutive
pressure gradient using only already-known time-level-$n$ data; the second is left
unlinearized, as discussed in Remark~\ref{rem:coupling-nonlinearity} below. To carry out
this linearization, we define the midpoint value
\[
\bar\rho^{n+\frac12}=\frac{\rho_f^{n+1}+\rho_f^n}{2},
\]
and set $\delta\rho = \rho_f^{n+1} - \rho_f^n$.
The key constitutive term is the midpoint approximation to
$\bar\rho^{n+1/2}\nabla\omega_\rho(\rho_f)$.
Since $\bar\rho^{n+1/2}\nabla\hat\omega_\rho$ is quadratic
in $\rho_f^{n+1}$ ($\hat\omega_\rho$ is the
first-order Taylor approximation of the divided difference), we expand and identify the $O(\delta\rho^2)$ cross-term:
\begin{align*}
\bar\rho^{n+\frac12}\nabla\hat{\omega}_{\rho}
&= \bar\rho^{n+\frac12}\nabla\omega_{\rho}(\rho_{f}^{n})
 + \rho_{f}^{n}\nabla\!\Bigl[\tfrac{1}{2}\omega_{\rho\rho}(\rho_{f}^{n})\,\delta\rho\Bigr]
 + \tfrac{1}{4}\,\delta\rho\,\nabla\!\bigl[\omega_{\rho\rho}(\rho_{f}^{n})\,\delta\rho\bigr].
\end{align*}
The last term is $O(\delta\rho^{2})$.
Dropping it gives the \emph{fully linearized} constitutive gradient,
which is linear in $\rho_f^{n+1}$:
\begin{equation}
\hat{\boldsymbol{d}}^{n+\frac12}
=\bar\rho^{n+\frac12}\nabla\omega_{\rho}(\rho_{f}^{n})
+\rho_{f}^{n}\nabla\!\left[\tfrac{1}{2}\omega_{\rho\rho}(\rho_{f}^{n})
\bigl(\rho_{f}^{n+1}-\rho_{f}^{n}\bigr)\right].
\label{eq:avg-domg}
\end{equation}
The dropped cross-term is $O(\delta\rho^{2})\sim O(\delta t^{2})$ per step,
at the level of the scheme's truncation error.
All individual terms in the density equation are centered at $t^{n+1/2}$,
so the scheme is second-order accurate for the density equation by the
standard midpoint-centering argument.

To stabilize the elastic stress in the momentum equation, we evaluate it as a weighted
combination of $\bu^{n+1}$ and $\bu^{n}$,
\begin{equation}
\sigma_{e,\alpha}\bigl(\bu^{n+1},\bu^{n}\bigr)
=\tfrac{1}{2}\sigma_{e}\!\bigl((1+\alpha_{\mathrm{stab}})\bu^{n+1}
+(1-\alpha_{\mathrm{stab}})\bu^{n}\bigr),
\label{eq:stab-stress}
\end{equation}
parameterized by $\alpha_{\mathrm{stab}}\geq 0$.
At $\alpha_{\mathrm{stab}}=0$ this is the midpoint rule
$\tfrac{1}{2}\sigma_e(\bu^{n+1}+\bu^{n})$;
at $\alpha_{\mathrm{stab}}=1$ it reduces to the fully implicit stress
$\sigma_e(\bu^{n+1})$.
In all numerical experiments we set $\alpha_{\mathrm{stab}}=\delta t$,
so that~\eqref{eq:stab-stress} differs from the midpoint rule by an
$O(\delta t)$ perturbation.

Finally, we arrive at the full discrete system:\\
Given $\left(\bu^{n},\rho_{f}^{n}\right)\in \mathcal{V}_{\bu}\times\mathcal{V}_{\rho}$, find
$\left(\bu^{n+1},\rho_{f}^{n+1}\right)\in\mathcal{V}_{\bu}\times\mathcal{V}_{\rho}$
such that for all test functions $(\psi_{\bu},\psi_{\rho})\in\mathcal{V}_{\bu}\times\mathcal{V}_{\rho}$,
\begin{align}
&\Bigl(\tfrac12\sigma_e\bigl((1+\alpha_{\mathrm{stab}})\bu^{n+1}+(1-\alpha_{\mathrm{stab}})\bu^n\bigr),\nabla\psi_{\bu}\Bigr)
+\alpha\bigl(\hat{\boldsymbol{d}}^{n+\frac12},\psi_{\bu}\bigr)
=\bigl(\boldsymbol{g}^{n+\frac12},\psi_{\bu}\bigr)_{\Gamma_s},\label{eq:2f-Ru}\\
&\bigl(\rho_f^{n+1}-\rho_f^n,\psi_{\rho}\bigr)
+\delta t\,\kappa\bigl(\hat{\boldsymbol{d}}^{n+\frac12},\nabla\psi_{\rho}\bigr)
-\alpha\bigl(\bar\rho^{n+\frac12}(\bu^{n+1}-\bu^n),\nabla\psi_{\rho}\bigr)
= 0.\label{eq:2f-Rrho}
\end{align}
Equation~\eqref{eq:2f-Ru} is linear jointly in $(\bu^{n+1},\rho_f^{n+1})$: the elastic
stress is linear in $\bu^{n+1}$, and $\hat{\boldsymbol{d}}^{n+\frac12}$ is linear in
$\rho_f^{n+1}$ alone (with all other coefficients frozen at time level $n$). In
\eqref{eq:2f-Rrho}, two of the three terms are likewise each linear in $\rho_f^{n+1}$
alone: the \emph{storage} term $\bigl(\rho_f^{n+1}-\rho_f^{n},\psi_\rho\bigr)$, which
accounts for the fluid mass accumulated over the step, and the \emph{Darcy} term
$\delta t\,\kappa\bigl(\hat{\boldsymbol{d}}^{n+\frac12},\nabla\psi_\rho\bigr)$, the
discrete counterpart of $\kappa\bigl(\nabla p(\rho_f),\nabla\psi_\rho\bigr)$
in~\eqref{eq:2f-weak-rho}, with the pressure gradient linearized
as in~\eqref{eq:avg-domg}. The coupling term $\bar\rho^{n+\frac12}(\bu^{n+1}-\bu^n)$,
however, contains the product $\rho_f^{n+1}\bu^{n+1}$ of the two unknowns and is
therefore bilinear in the pair $(\bu^{n+1},\rho_f^{n+1})$ jointly. The assembled
system~\eqref{eq:2f-Ru}--\eqref{eq:2f-Rrho} is consequently not fully linear.

\begin{remark}[Residual nonlinearity from the coupling term]\label{rem:coupling-nonlinearity}
Although the constitutive pressure term $\hat{\boldsymbol{d}}^{n+\frac12}$ is linearized
exactly as described above, the displacement--density coupling term
$\bar\rho^{n+\frac12}(\bu^{n+1}-\bu^n)$ retains the bilinear product
$\rho_f^{n+1}\bu^{n+1}$ and is not linearized. The system~\eqref{eq:2f-Ru}--\eqref{eq:2f-Rrho}
is therefore solved by Newton's method~\citep{kelley1995iterative} rather than a single
direct linear solve. This
residual nonlinearity is mild: across all footing configurations in
Section~\ref{subsec:footing}, Newton's method converges in one or two iterations per
time step, a small overhead compared to the fully nonlinear system that
would result from leaving $\hat{\boldsymbol{d}}^{n+\frac12}$ unlinearized as well.
\end{remark}

\subsection{Discrete energy identity}\label{par:2f-stab-energy}

The continuous model is dissipative by construction (Section~\ref{sec:model-derivation}):
energy decreases monotonically except for work done by external tractions. A
discretization that only approximately preserves this structure can, over many time
steps, either lose energy too fast or manufacture energy it should not have.  This is an
especially serious risk in the stiff, externally forced footing problems of
Section~\ref{sec:numerical-experiments}, run for thousands of time steps. The following
theorem shows that the scheme~\eqref{eq:2f-Ru}--\eqref{eq:2f-Rrho}
inherits the continuous dissipation structure exactly, up to a computable,
higher-order defect, and is therefore unconditionally energy-stable for any $\delta t$.

\begin{theorem}[Discrete energy identity]\label{thm:energy-identity}
At each step $n\geq 0$,
\begin{equation}
\mathcal{E}^{n+1}-\mathcal{E}^{n}
+\delta t\,D_{\mathrm{pred}}^{n}
+\alpha_{\mathrm{stab}}\!\int_\Omega\!W_{e}\bigl(I+\nabla(\bu^{n+1}-\bu^{n})\bigr)\,dx
=\bigl(\boldsymbol{g}^{n+\frac12},\bu^{n+1}-\bu^{n}\bigr)_{\Gamma_s}+\mathrm{Defect}^{n},
\label{eq:disc-2f-stab-energy}
\end{equation}
where
\[
\mathcal{E}^n=\int_\Omega W_{e}(I+\nabla\bu^n)+\omega(\rho_f^n)\,dx,
\qquad
D_{\mathrm{pred}}^{n}=\kappa\!\int_\Omega\frac{\bigl|\hat{\boldsymbol{d}}^{n+\frac12}\bigr|^2}{\bar\rho^{n+\frac12}}\,dx\geq0,
\]
and the defect has the explicit closed form
\begin{equation}
\mathrm{Defect}^{n}
=\int_\Omega \boldsymbol{r}^{n+\frac12}\cdot
\Bigl[\alpha(\bu^{n+1}-\bu^n)-\delta t\,\kappa\,\frac{\hat{\boldsymbol{d}}^{n+\frac12}}{\bar\rho^{n+\frac12}}\Bigr]\,dx.
\label{eq:defect-explicit}
\end{equation}
Here,
\begin{equation}
\boldsymbol{r}^{n+\frac12}
:=\frac{\delta\rho}{4}\,\nabla\!\bigl[\omega_{\rho\rho}(\rho_f^n)\,\delta\rho\bigr]
+\frac{\bar\rho^{n+\frac12}}{6}\,\nabla\!\bigl[\omega_{\rho\rho\rho}(\xi)\,(\delta\rho)^2\bigr],
\label{eq:r-defect}
\end{equation}
with $\xi=\xi(x)$ an intermediate value between $\rho_f^n(x)$ and $\rho_f^{n+1}(x)$
given by Taylor's theorem with Lagrange remainder (assuming $\omega\in C^3$).
Since $\delta\rho=\rho_f^{n+1}-\rho_f^{n}=O(\delta t)$ and $\bu^{n+1}-\bu^n=O(\delta t)$,
$\boldsymbol{r}^{n+\frac12}=O(\delta t^2)$ and $\mathrm{Defect}^n$ is $O(\delta t^3)$ per step.
\end{theorem}

\begin{proof}
The proof proceeds by testing the momentum equation~\eqref{eq:2f-Ru} and the density
equation~\eqref{eq:2f-Rrho} with the increment $\bu^{n+1}-\bu^n$ and the exact divided
difference $\tilde\omega_\rho$, respectively, then combining the two resulting
identities.

\paragraph{Momentum test.}
Choose $\psi_{\bu}=\bu^{n+1}-\bu^{n}$ in~\eqref{eq:2f-Ru}.
Using linearity of $\sigma_e$ and the symmetry
$(\sigma_e(\boldsymbol{w}),\nabla\boldsymbol{v})=(\sigma_e(\boldsymbol{v}),\nabla\boldsymbol{w})$,
\begin{align*}
&\tfrac12\bigl(\sigma_e\bigl((1+\alpha_\mathrm{stab})\bu^{n+1}
  +(1-\alpha_\mathrm{stab})\bu^n\bigr),\nabla(\bu^{n+1}-\bu^n)\bigr)\\
&\quad=\underbrace{\tfrac12\bigl(\sigma_e(\bu^{n+1}+\bu^n),
  \nabla(\bu^{n+1}-\bu^n)\bigr)}_{=\,\mathcal{E}_e^{n+1}-\mathcal{E}_e^n}
+\underbrace{\tfrac{\alpha_\mathrm{stab}}{2}\bigl(\sigma_e(\bu^{n+1}-\bu^n),
  \nabla(\bu^{n+1}-\bu^n)\bigr)}_{=\,\alpha_\mathrm{stab}\,\mathcal{S}^n\,\geq\,0}.
\end{align*}
Here $\mathcal{E}_e^n=\int_\Omega W_e(I+\nabla\bu^n)\,dx$;
the identity $\mathcal{E}_e^{n+1}-\mathcal{E}_e^n
=\tfrac12(\sigma_e(\bu^{n+1}+\bu^n),\nabla(\bu^{n+1}-\bu^n))$
follows from $W_e(\bu)=\tfrac12(\sigma_e(\bu),\nabla\bu)$ by bilinearity, and
$\mathcal{S}^n=\int_\Omega W_e(I+\nabla(\bu^{n+1}-\bu^n))\,dx\geq0$
by coercivity of $\sigma_e$.
The full momentum test gives
\begin{equation}
\mathcal{E}_e^{n+1}-\mathcal{E}_e^n
+\alpha_\mathrm{stab}\,\mathcal{S}^n
+\alpha\bigl(\hat{\boldsymbol{d}}^{n+\frac12},\bu^{n+1}-\bu^n\bigr)
=\bigl(\boldsymbol{g}^{n+\frac12},\bu^{n+1}-\bu^n\bigr)_{\Gamma_s}.
\label{eq:mom-test-eq}
\end{equation}

\paragraph{Density test.}
Choose $\psi_\rho=\tilde\omega_\rho$, the exact divided difference of $\omega$
defined by $(\rho_f^{n+1}-\rho_f^{n})\,\tilde\omega_\rho=\omega(\rho_f^{n+1})-\omega(\rho_f^{n})$,
in~\eqref{eq:2f-Rrho}, and process each of the three terms.

\emph{Storage term $(\delta\rho,\tilde\omega_\rho)$.}
By definition of $\tilde\omega_\rho$,
\begin{equation}
(\delta\rho,\tilde\omega_\rho)
=\int_\Omega\bigl[\omega(\rho_f^{n+1})-\omega(\rho_f^n)\bigr]\,dx
=\mathcal{E}_f^{n+1}-\mathcal{E}_f^n,
\quad \mathcal{E}_f^n:=\int_\Omega\omega(\rho_f^n)\,dx,
\label{eq:storage-exact}
\end{equation}
exactly, with no remainder.

\emph{The flux mismatch $\boldsymbol{r}^{n+\frac12}$.}
By Taylor's theorem with Lagrange remainder,
$\tilde\omega_\rho=\omega_\rho(\rho_f^n)+\tfrac12\omega_{\rho\rho}(\rho_f^n)\,\delta\rho
+\tfrac16\omega_{\rho\rho\rho}(\xi)\,(\delta\rho)^2$ exactly, for the same
$\xi=\xi(x)$ as in~\eqref{eq:r-defect}. Writing
$\nabla\tilde\omega_\rho=\nabla\omega_\rho(\rho_f^n)+\tfrac12\nabla[\omega_{\rho\rho}(\rho_f^n)\,\delta\rho]
+\tfrac16\nabla[\omega_{\rho\rho\rho}(\xi)\,(\delta\rho)^2]$
and using $\rho_f^n-\bar\rho^{n+\frac12}=-\tfrac{\delta\rho}{2}$,
the definition~\eqref{eq:avg-domg} gives the key decomposition
\begin{equation}
\hat{\boldsymbol{d}}^{n+\frac12}
=\bar\rho^{n+\frac12}\,\nabla\tilde\omega_\rho-\boldsymbol{r}^{n+\frac12},
\label{eq:d-decomp-eq}
\end{equation}
with $\boldsymbol{r}^{n+\frac12}$ as in~\eqref{eq:r-defect}.

\emph{Dissipation term $\delta t\,\kappa(\hat{\boldsymbol{d}}^{n+\frac12},\nabla\tilde\omega_\rho)$.}
Since $\nabla\tilde\omega_\rho-\hat{\boldsymbol{d}}^{n+\frac12}/\bar\rho^{n+\frac12}
=\boldsymbol{r}^{n+\frac12}/\bar\rho^{n+\frac12}$ by~\eqref{eq:d-decomp-eq},
\begin{equation}
\delta t\,\kappa\bigl(\hat{\boldsymbol{d}}^{n+\frac12},\nabla\tilde\omega_\rho\bigr)
=\delta t\,D_{\mathrm{pred}}^n+D_2^n,
\qquad
D_2^n:=\delta t\,\kappa\Bigl(\hat{\boldsymbol{d}}^{n+\frac12},
\frac{\boldsymbol{r}^{n+\frac12}}{\bar\rho^{n+\frac12}}\Bigr).
\label{eq:D2-eq}
\end{equation}

\emph{Coupling term $-\alpha(\bar\rho^{n+\frac12}(\bu^{n+1}-\bu^n),\nabla\tilde\omega_\rho)$.}
Rearranging the integrand,
$-\alpha(\bar\rho^{n+\frac12}(\bu^{n+1}-\bu^n),\nabla\tilde\omega_\rho)
=-\alpha(\bar\rho^{n+\frac12}\nabla\tilde\omega_\rho,\bu^{n+1}-\bu^n)$.

Collecting all three terms from the density test:
\begin{equation}
\mathcal{E}_f^{n+1}-\mathcal{E}_f^n
+\delta t\,D_{\mathrm{pred}}^n
-\alpha\bigl(\bar\rho^{n+\frac12}\nabla\tilde\omega_\rho,\bu^{n+1}-\bu^n\bigr)
=-D_2^n.
\label{eq:rho-test-eq}
\end{equation}

\paragraph{Coupling cancellation.}
Adding~\eqref{eq:mom-test-eq} and~\eqref{eq:rho-test-eq},
the coupling contributions combine as
$\alpha(\hat{\boldsymbol{d}}^{n+\frac12}-\bar\rho^{n+\frac12}\nabla\tilde\omega_\rho,\,\bu^{n+1}-\bu^n)$.
Using~\eqref{eq:d-decomp-eq}, this residual is
\begin{equation}
D_3^n
:=\alpha\bigl(\hat{\boldsymbol{d}}^{n+\frac12}-\bar\rho^{n+\frac12}\nabla\tilde\omega_\rho,\,
\bu^{n+1}-\bu^n\bigr)
=-\alpha\bigl(\boldsymbol{r}^{n+\frac12},\bu^{n+1}-\bu^n\bigr),
\label{eq:D3-eq}
\end{equation}
which is $O(\delta t^3)$ per step since $\boldsymbol{r}^{n+\frac12}=O(\delta t^2)$
and $\bu^{n+1}-\bu^n=O(\delta t)$.

\paragraph{Final assembly.}
Summing~\eqref{eq:mom-test-eq} and~\eqref{eq:rho-test-eq}
and writing $D_3^n$ for the coupling residual gives
\[
\mathcal{E}^{n+1}-\mathcal{E}^n
+\delta t\,D_{\mathrm{pred}}^n
+\alpha_\mathrm{stab}\,\mathcal{S}^n
+D_2^n+D_3^n
=\bigl(\boldsymbol{g}^{n+\frac12},\bu^{n+1}-\bu^n\bigr)_{\Gamma_s}.
\]
Setting $\mathrm{Defect}^n=-(D_2^n+D_3^n)$ and
expanding via~\eqref{eq:D2-eq}--\eqref{eq:D3-eq}
yields identity~\eqref{eq:disc-2f-stab-energy}
with explicit defect~\eqref{eq:defect-explicit}: since $D_2^n$ and $D_3^n$
both stem from the same flux mismatch $\boldsymbol{r}^{n+\frac12}$, they
combine into the single integral shown there rather than requiring two
separate terms.

Both dissipative terms are non-negative:
$\delta t\,D_{\mathrm{pred}}^n\geq 0$ since $\bar\rho^{n+\frac12}>0$, $\kappa>0$,
and $|\hat{\boldsymbol{d}}^{n+\frac12}|^2\geq0$, and
$\alpha_\mathrm{stab}\,\mathcal{S}^n\geq 0$ by coercivity of $\sigma_e$.
The scheme is therefore unconditionally energy-stable for any $\alpha_\mathrm{stab}\geq 0$.
\end{proof}

\begin{remark}[Stabilization order]
For $\alpha_\mathrm{stab}=\delta t$ the stabilization term is $O(\delta t^3)$ per step,
the same order as $\mathrm{Defect}^n$,
so it does not degrade the cumulative $O(\delta t^2)$ energy accuracy,
but adds physical-like elastic damping proportional to
$\int_\Omega W_{e}(I+\nabla(\bu^{n+1}-\bu^{n}))\,dx$.
\end{remark}

\section{Numerical experiments}
\label{sec:numerical-experiments}

The purpose of the experiments in this section is twofold: (1) to verify that
the two-field discretization reproduces the accuracy and energy-law
behavior proved in Section~\ref{sec:twofield-discretization} on a
representative, analytically tractable configuration, and (2) to demonstrate
the resulting model in a practically relevant application (consolidation under a surface load) across the lateral-boundary
treatments and fluid-compressibility exponents that are the model's
principal new degrees of freedom. A systematic numerical-analysis study of
the discretization itself is left to future work; the goal here is to establish that the
model and its scheme are correct and usable, not to characterize the scheme's
numerical behavior exhaustively.

\subsection{Numerical setup}\label{subsec:numerical-setup}

All experiments use a two-dimensional unit square domain, $(0,1)^{2}$,
discretized by a uniform $K\times K$ mesh of right triangles
(each square cell split diagonally into two triangles),
giving mesh resolution $h=1/K$; the time step $\delta t$ is specified per subsection.
The Biot coupling coefficient and fluid modulus are fixed at $\alpha=1$ and $M=1$
throughout.
Elastic parameters are expressed via Young's modulus $E$ and Poisson's ratio $\nu$:
\[
\mu = \frac{E}{2(1+\nu)},\qquad
\lambda = \frac{E\nu}{(1+\nu)(1-2\nu)}.
\]
Two parameter sets are used across the experiments:
\begin{itemize}
\item \textbf{Benchmark parameters} (accuracy and energy tests): $E=1$, $\nu=0.2$
  ($\mu=5/12\approx0.417$, $\lambda=5/18\approx0.278$), $\kappa=1$.
\item \textbf{Footing parameters} (externally forced problems):
  $E=1000$, $\nu=0.2$
  ($\mu\approx416.7$, $\lambda\approx277.8$), $\kappa=10^{-2}$.
  The higher stiffness and lower permeability place the footing problem
  in an under-drained consolidation regime that is physically more
  demanding and practically more relevant.
\end{itemize}
Both parameter sets use the moderate value $\nu=0.2$, which keeps the
displacement discretization away from the near-incompressible locking that
standard Galerkin finite elements for elasticity can exhibit as $\nu\to0.5$; extending the
method to nearly-incompressible biological or clay-rich geomechanical
materials in that limit is left to future work.

For $(\bu,\rho_f)$ we use $\bm{\mathcal{P}}_2\times\mathcal{P}_1$: vector Lagrange
degree~2 for displacement and continuous Lagrange degree~1 for density, matching the
regularity required by the weak forms~\eqref{eq:2f-weak-u}--\eqref{eq:2f-weak-rho}
($\bu\in\mathbf H^1$, $\rho_f\in H^1$) while keeping the density unknown one polynomial
degree below displacement, consistent with the discretization used throughout
Section~\ref{sec:twofield-discretization}.

The time discretization and linearization are as described in
Section~\ref{subsec:2f-scheme}; boundary conditions are imposed as in
Section~\ref{sec:model-derivation}, with essential (Dirichlet) conditions built into
$\mathcal V_{\bu}$ and $\mathcal V_\rho$ and traction and no-flux conditions entering
naturally through the weak forms.

As discussed in Remark~\ref{rem:coupling-nonlinearity}, the assembled system retains a
mild bilinear nonlinearity through the displacement--density coupling term and is therefore
solved by Newton's method; observed iteration counts are reported in
Section~\ref{subsec:footing}. Each Newton step solves a linear system whose matrix is
the Jacobian of the residuals with respect to the unknowns
$(\bu^{n+1},\rho_f^{n+1})$~\citep{kelley1995iterative}. Writing $\mathcal{R}_{\bu}$ and
$\mathcal{R}_{\rho}$ for the residuals of~\eqref{eq:2f-Ru} and~\eqref{eq:2f-Rrho}, that
Jacobian is the $2\times2$ block matrix
\begin{equation}
\begin{bmatrix}
\dfrac{\partial\mathcal{R}_{\bu}}{\partial\bu^{n+1}} &
\dfrac{\partial\mathcal{R}_{\bu}}{\partial\rho_f^{n+1}}\\[2.2ex]
\dfrac{\partial\mathcal{R}_{\rho}}{\partial\bu^{n+1}} &
\dfrac{\partial\mathcal{R}_{\rho}}{\partial\rho_f^{n+1}}
\end{bmatrix}
=
\begin{bmatrix}
A_{\bu\bu} & A_{\bu\rho}\\[0.6ex]
A_{\rho\bu} & A_{\rho\rho}
\end{bmatrix},
\label{eq:jacobian}
\end{equation}
in which $A_{\bu\bu}$ is an elasticity operator and $A_{\bu\rho}$ the linearized
constitutive term $\hat{\boldsymbol{d}}^{n+\frac12}$ of the momentum residual. The
bilinear coupling term of Remark~\ref{rem:coupling-nonlinearity} contributes to the
remaining two blocks at once: differentiating
$-\alpha\bigl(\bar\rho^{n+\frac12}(\bu^{n+1}-\bu^{n}),\nabla\psi_\rho\bigr)$ with
respect to $\bu^{n+1}$ gives $A_{\rho\bu}$, while differentiating it with respect to
$\rho_f^{n+1}$, through $\bar\rho^{n+\frac12}$, contributes
$-\tfrac{\alpha}{2}\bigl((\bu^{n+1}-\bu^{n})\,\cdot\,,\nabla\psi_\rho\bigr)$ to
$A_{\rho\rho}$. The diagonal block $A_{\rho\rho}$ is therefore not a pure
storage-plus-diffusion operator: it depends on the current displacement iterate, which
is why the Jacobian must be reassembled at every Newton step.
That solve is preconditioned block-diagonally; written with exact block inverses, the
preconditioner is
\[
\mathcal P_{\mathrm{exact}}^{-1}=
\operatorname{diag}\big(A_{\bu\bu}^{-1},\,A_{\rho\rho}^{-1}\big).
\]
The off-diagonal blocks are therefore retained in the Jacobian that Newton's method
uses, but dropped from the preconditioner; the coupling contribution sitting inside
$A_{\rho\rho}$ is kept. Neither diagonal block is inverted exactly in practice: instead,
$A_{\bu\bu}^{-1}$ is approximated by one algebraic-multigrid
V-cycle~\citep{falgout2002hypre,henson2002boomeramg}, and $A_{\rho\rho}^{-1}$ by a
diagonal preconditioner~\citep{saad2003iterative}, which is well conditioned at the
time-step sizes used here. The resulting preconditioner is applied within a GMRES
iteration~\citep{saad1986gmres}.

All discrete systems are assembled and solved using the FEniCSx finite-element
library~\citep{baratta2023dolfinx}, with PETSc~\citep{balay1997petsc} as the linear- and
nonlinear-algebra backend; the block preconditioner above is accessed through PETSc's
field-split interface within its Newton solver.

Accuracy in Section~\ref{subsec:mfld-accuracy} is measured at the final simulation time
$T$ in the quantities the model's own energy structure controls, rather than in a
generic $L^2$ norm. For the displacement, we use the elastic energy norm of the error
$\boldsymbol{e}_{\bu}=\bu^{N}-\bu^{*}(T)$,
\begin{equation}
\|\boldsymbol{e}_{\bu}\|_{E}
:=\Bigl(\int_\Omega\sigma_{e}(\boldsymbol{e}_{\bu}):\nabla\boldsymbol{e}_{\bu}\,dx\Bigr)^{1/2},
\label{eq:err-u}
\end{equation}
the quadratic form associated with the elastic energy whose increments drive
Theorem~\ref{thm:energy-identity}. This is the energy norm in the usual finite-element
sense~\citep{brenner2008mathematical}, induced by the elastic bilinear form
itself, and it is equivalent to the $\mathbf{H}^{1}$ seminorm by coercivity
of $\sigma_{e}$ together with Korn's inequality~\citep{ciarlet1988elasticity}. Measuring the displacement error in $\|\cdot\|_{E}$ therefore reports
how much elastic energy the discretization misplaces, in the same units the stability
result controls. The corresponding quantity for the density is not a norm but the
Bregman divergence~\citep{bregman1967relaxation} generated by the fluid free energy
$\omega$, evaluated on $e_{\rho}=\rho_{f}^{N}-\rho_{f}^{*}(T)$,
\begin{equation}
D_{\omega}\bigl(\rho_{f}^{N}\,\|\,\rho_{f}^{*}\bigr)
=\omega(\rho_{f}^{N})-\omega(\rho_{f}^{*})-\omega_{\rho}(\rho_{f}^{*})\,e_{\rho}
\;\geq\;0,
\label{eq:err-rho}
\end{equation}
nonnegative by convexity of $\omega$ and vanishing only when $\rho_{f}^{N}=\rho_{f}^{*}$.
We report $\bigl(\int_\Omega D_{\omega}\,dx\bigr)^{1/2}$, which for small errors behaves
like $\bigl(\tfrac12\int_\Omega\omega_{\rho\rho}(\rho_{f}^{*})\,e_{\rho}^{2}\,dx\bigr)^{1/2}$,
an $L^{2}$ norm weighted by $1/\rho_{f}^{*}$ under the ideal-gas closure --- the same
weighting carried by $D_{\mathrm{pred}}^{n}$ in~\eqref{eq:disc-2f-stab-energy}.
Quantifying the distance between two states of a system by the divergence generated by
its own convex free energy, rather than by a norm chosen independently of it, is the
standard device of entropy methods for diffusive
equations~\citep{jungel2016entropy}; for the ideal-gas closure
$\omega=M\rho_f\ln\rho_f$ used in this test, $D_{\omega}$ reduces to the relative
entropy of $\rho_{f}^{N}$ with respect to $\rho_{f}^{*}$,
\[
D_{\omega}\bigl(\rho_{f}^{N}\,\|\,\rho_{f}^{*}\bigr)
=M\Bigl[\rho_{f}^{N}\ln\frac{\rho_{f}^{N}}{\rho_{f}^{*}}-\rho_{f}^{N}+\rho_{f}^{*}\Bigr].
\]
Measuring error in these two quantities tests the discretization in precisely the
structure it is designed to preserve: together they are the two halves of the stored
energy $\mathcal{E}$ whose evolution Theorem~\ref{thm:energy-identity} governs.
The two are of different types.  The displacement measure controls a gradient, the
density measure does not.  However,  that asymmetry is inherited from the energy itself,
which is quadratic in $\nabla\bu$ for the solid and pointwise in $\rho_{f}$ for the
fluid; it is not a modelling choice made for the convergence study.
With errors computed at three mesh resolutions (or three time steps), we report the two
successive-refinement rates $\log(e_i/e_{i+1})/\log(h_i/h_{i+1})$ between each
consecutive pair, which is why two rates, not one, are quoted for each field.

\subsection{Accuracy: manufactured solution}\label{subsec:mfld-accuracy}

Consider the simplified system
\[
\begin{cases}
-\nabla\cdot\sigma_{e}\left(\bu\right)+\alpha\nabla p=f_{1}\left(x,\,t\right),\\
\boldsymbol{q}=-\kappa\nabla p,\\
\frac{1}{M}\partial_{t}p+\nabla\cdot\left(\frac{\alpha}{M}p\partial_{t}\bu\right)+\nabla\cdot\boldsymbol{q}=f_{3}\left(x,\,t\right),
\end{cases}
\]
 with
\[
\sigma_{e}\left(\bu\right)=\mu\left(\nabla\bu+\nabla\bu^{T}\right)+\lambda\left(\nabla\cdot\bu\right)I.
\]

Here the goal is to evaluate convergence with a known exact solution that
exercises genuine displacement--density coupling through the term
$\alpha\rho_f\partial_t\bu$, rather than decoupling it by holding $\bu$
stationary. We introduce an extra forcing term and take the exact
solution to be
\begin{align*}
p= & 2-e^{-2\pi^{2}\kappa Mt}\cos\left(\pi x\right)\cos\left(\pi y\right),\\
\bu= & \varepsilon\sin(2\pi t)\,\bu_{0}(x,y),\qquad
\bu_{0}(x,y) = \left[\begin{array}{c}
\sin\left(2\pi x\right)\left(1-\cos\left(2\pi y\right)\right)\\
\sin\left(2\pi y\right)\left(1-\cos\left(2\pi x\right)\right)
\end{array}\right],
\end{align*}
with $\varepsilon=0.1$, so that $\partial_t\bu=\varepsilon\,2\pi\cos(2\pi t)\,\bu_0$
is nonzero throughout the test; since $\cos(2\pi t)\approx1$ at the short
times used below, the coupling term is $O(1)$, not vanishingly small, from
the start of the simulation.
Under the ideal-gas closure $p=M\rho_f$ used for this test, the manufactured
density is $\rho_f^*=p^*/M$.
Then, the force exerted on the structure is calculated accordingly.

 Since $\sigma_e$ is linear, it suffices to compute the elastic stress
generated by the spatial profile $\bu_0$ alone; the momentum forcing is
then this expression scaled by $\varepsilon\sin(2\pi t)$. In the
computation below, $\bu$ denotes $\bu_0$ for brevity. First, we
calculate the elastic stress and its divergence,
\begin{align*}
\sigma_{e}\left(\bu\right)= & \mu\left(\nabla\bu+\nabla\bu^{T}\right)+\lambda\left(\nabla\cdot\bu\right)I\\
= & \mu\left[\begin{array}{cc}
	2\frac{\partial\bu_{1}}{\partial x} &
	\frac{\partial\bu_{1}}{\partial y}+\frac{\partial\bu_{2}}{\partial x}\\
	\frac{\partial\bu_{1}}{\partial y}+\frac{\partial\bu_{2}}{\partial x} &
	2\frac{\partial\bu_{2}}{\partial y}
\end{array}\right]
+\lambda\left(\nabla\cdot\bu\right)I\\
= & 4\pi\mu\left[\begin{array}{cc}
	\cos\left(2\pi x\right)\left(1-\cos\left(2\pi y\right)\right) &
	\sin\left(2\pi x\right)\sin\left(2\pi y\right)\\
	\sin\left(2\pi x\right)\sin\left(2\pi y\right)&
	\cos\left(2\pi y\right)\left(1-\cos\left(2\pi x\right)\right)
\end{array}\right]\\
&+2\pi\lambda\left(\cos\left(2\pi x\right)+\cos\left(2\pi y\right)-2\cos\left(2\pi x\right)\cos\left(2\pi y\right)\right)I,\\[2ex]
	\nabla\cdot \sigma_{e}\left(\bu\right)= & \mu\nabla\cdot\left(\nabla\bu+\nabla\bu^{T}\right)+\lambda\nabla\cdot\left[\left(\nabla\cdot\bu\right)I\right]\\
	= & \mu\left[\begin{array}{c}
		2\frac{\partial^2\bu_{1}}{\partial x^2} +
		\frac{\partial^2\bu_{1}}{\partial y^2}+\frac{\partial^2\bu_{2}}{\partial x\partial y}\\
		\frac{\partial^2\bu_{1}}{\partial x\partial y}+\frac{\partial^2\bu_{2}}{\partial x^2} +
		2\frac{\partial^2\bu_{2}}{\partial y^2}
	\end{array}\right]
	+\lambda\left(\frac{\partial\nabla\cdot\bu}{\partial x}+\frac{\partial\nabla\cdot\bu}{\partial y}\right)\\
	= & \left(8\pi^2\mu+4\pi^2\lambda\right)\left[\begin{array}{cc}
		\sin\left(2\pi x\right)\left(2\cos\left(2\pi y\right)-1\right) \\
		\sin\left(2\pi y\right)\left(2\cos\left(2\pi x\right)-1\right)
	\end{array}\right].
\end{align*}
Then, the gradient pressure term is given by
\[\alpha\nabla p=  \alpha\pi e^{-2\pi^{2}\kappa Mt}
\left[\begin{array}{c}
	\sin\left(\pi x\right)\cos\left(\pi y\right)\\
	\cos\left(\pi x\right)\sin\left(\pi y\right)
\end{array}\right].\]

Since $\bu=\varepsilon\sin(2\pi t)\,\bu_0$ and $\sigma_e$ is linear,
$-\nabla\cdot\sigma_e(\bu)=-\varepsilon\sin(2\pi t)\,\nabla\cdot\sigma_e(\bu_0)$;
combining this with the pressure-gradient term above gives the momentum
forcing
\[
f_1=-\varepsilon\sin(2\pi t)\left(8\pi^{2}\mu+4\pi^{2}\lambda\right)
\left[\begin{array}{c}
	\sin\left(2\pi x\right)\left(2\cos\left(2\pi y\right)-1\right)\\
	\sin\left(2\pi y\right)\left(2\cos\left(2\pi x\right)-1\right)
\end{array}\right]
+\alpha\pi e^{-2\pi^{2}\kappa Mt}
\left[\begin{array}{c}
	\sin\left(\pi x\right)\cos\left(\pi y\right)\\
	\cos\left(\pi x\right)\sin\left(\pi y\right)
\end{array}\right].
\]

On the other hand, the pressure was chosen to be a solution of the heat equation
\[
\frac1M\partial_tp=\kappa\Delta p,
\]
independently of the time profile chosen for $\bu$; the momentum equation
only requires $-\nabla\cdot\sigma_e(\bu)+\alpha\nabla p=f_1$ to hold at each
instant. Being genuinely time-dependent, however, $\bu$ introduces a source
term into the mass equation, since $\partial_t\bu\neq0$:
\[
f_3=\frac{\alpha}{M}\nabla\cdot\!\left(p\,\partial_t\bu\right)
=\frac{\alpha}{M}\,\varepsilon\,2\pi\cos(2\pi t)
\Bigl[p\,\nabla\cdot\bu_0+\nabla p\cdot\bu_0\Bigr].
\]
Both terms in $f_3$ are required: the second, $\nabla p\cdot\bu_0$,
integrates to zero over the domain and so would be invisible to a check of
global mass conservation alone, but it does not vanish pointwise and must
be included for the manufactured solution to be exact.

\begin{remark}[Choice of time step for the spatial convergence test.]
Because $p^*$ satisfies the pure diffusion equation, the characteristic
decay time of the manufactured solution is $T_{\mathrm{decay}}=1/(\kappa\pi^2)$,
which is $\approx0.10$ at the benchmark value $\kappa=1$ used here.
The temporal discretization error for the density is proportional to
$\delta t\times\kappa\times\|\partial_t\rho^*\|$, so to isolate the
spatial error $O(h^2)$ the time step must satisfy
$\delta t\ll h^2/(\kappa\pi^2)$; at the finest mesh used below,
$K=32$ ($h=1/32$), this requires $\delta t\lesssim5\times10^{-5}$.
The value $\delta t=10^{-6}$ used for the spatial convergence experiment is well
inside that bound.
\end{remark}

Figure~\ref{fig:accuracy-2f} reports the results.
Under this genuine displacement--density coupling, the two-field method achieves $O(h^2)$
spatial convergence in both displacement (rates $2.04$, $2.01$ --- the two
successive-refinement rates from the three mesh resolutions, as explained in
Section~\ref{subsec:numerical-setup}) and density (rates $2.02$, $2.00$).
Each field thus attains the optimal rate its space admits in the measure the energy
structure assigns to it: $O(h^2)$ for the $\bm{\mathcal P}_2$ displacement in the
energy norm~\eqref{eq:err-u}, which controls a gradient, and $O(h^2)$ for the
$\mathcal P_1$ density in~\eqref{eq:err-rho}, which does not.
Temporally, at fixed $K=256$ --- chosen so that the spatial error stays well below the
temporal error across the $\delta t$ range shown --- the method exhibits
$O(\delta t^2)$ convergence for both fields (rates $1.99$, $1.89$ for $\bu$; $2.05$,
$2.22$ for $\rho_f$), unaffected by the coupling; errors throughout are measured at the
final time $T$ in the energy quantities defined in
Section~\ref{subsec:numerical-setup}.

\pgfplotsset{
  acc/.style={
    width=0.47\textwidth,height=0.36\textwidth,
    grid=both,grid style={gray!20,line width=0.3pt},
    mark size=2.5pt,thick,
    legend style={font=\small,inner sep=3pt,row sep=-1pt},
    label style={font=\small},tick label style={font=\small},
    title style={font=\small},
  }
}
\begin{figure}[h!]
\centering
\begin{tikzpicture}
\begin{loglogaxis}[acc,
  title={(a) Spatial Error ($\delta t=10^{-6}$)},
  xlabel={$h$},ylabel={energy error},
  xmin=0.02,xmax=0.18,ymin=1.5e-4,ymax=2e-2,
  legend pos=north west,
]
\addplot[blue,mark=o] coordinates
  {(0.125,1.048e-2)(0.0625,2.548e-3)(0.03125,6.314e-4)};
\addlegendentry{$\bu$}
\addplot[red,mark=square] coordinates
  {(0.125,4.521e-3)(0.0625,1.118e-3)(0.03125,2.786e-4)};
\addlegendentry{$\rho_f$}
\addplot[black,dashed,domain=0.025:0.14,samples=2]{0.44*x^2};
\addlegendentry{$O(h^2)$}
\end{loglogaxis}
\end{tikzpicture}
\hfill
\begin{tikzpicture}
\begin{loglogaxis}[acc,
  title={(b) Temporal Error ($K{=}256$)},
  xlabel={$\delta t$},ylabel={energy error},
  xmin=1.5e-3,xmax=1.5e-2,ymin=1.5e-5,ymax=2e-3,
  legend pos=north west,
]
\addplot[blue,mark=o] coordinates
  {(0.01,8.499e-4)(0.005,2.146e-4)(0.0025,5.794e-5)};
\addlegendentry{$\bu$}
\addplot[red,mark=square] coordinates
  {(0.01,5.073e-4)(0.005,1.226e-4)(0.0025,2.629e-5)};
\addlegendentry{$\rho_f$}
\addplot[black,dashed,domain=2e-3:1.2e-2,samples=2]{6.6*x^2};
\addlegendentry{$O(\delta t^2)$}
\end{loglogaxis}
\end{tikzpicture}
\caption{Two-field manufactured-solution accuracy under genuine
  displacement--density coupling ($E=1$, $\nu=0.2$, $M=\alpha=1$;
  $\bu^*=\varepsilon\sin(2\pi t)\,\bu_0(\bx)$ with $\varepsilon=0.1$, so
  $\partial_t\bu^*\neq0$ throughout). Errors are measured at $t=T$ in the energy
  quantities of Section~\ref{subsec:numerical-setup}: the elastic energy
  norm~\eqref{eq:err-u} for $\bu$, and $(\int_\Omega D_\omega\,dx)^{1/2}$
  from~\eqref{eq:err-rho} for $\rho_f$.
  \textbf{(a)}~Spatial convergence, both fields $O(h^2)$, at $\delta t=10^{-6}$,
  $\kappa=1$, $T=5\times10^{-6}$ (chosen so $\delta t\ll h^2/(\kappa\pi^2)$ at
  $K=32$, keeping temporal error below spatial).
  \textbf{(b)}~Temporal convergence ($\alpha_{\mathrm{stab}}=\delta t$), both fields
  $O(\delta t^2)$, at $K=256$, $T=0.1$. See body text for rates and discussion.}
\label{fig:accuracy-2f}
\end{figure}
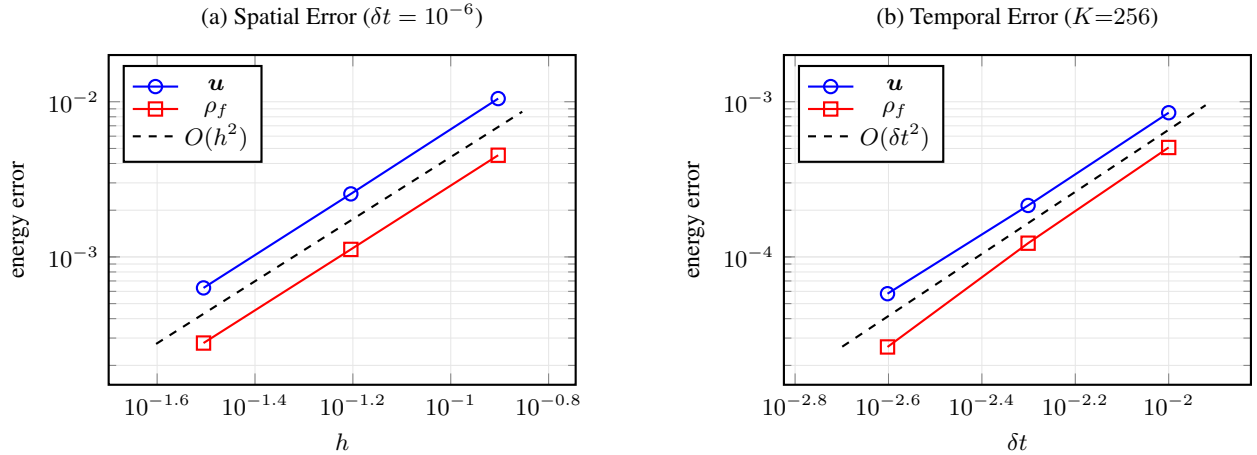

\subsection{Energy behavior}\label{subsec:energy-2f}

To study the energy behavior of the formulation, we consider boundary conditions where all four walls are clamped ($\bu=\boldsymbol{0}$ on $\partial\Omega$)
and no external forcing is applied (see Figure~\ref{fig:blob-schematic}).
The natural boundary condition enforces
$\boldsymbol{j}\cdot\boldsymbol{n}=0$ (total no-flux) on all walls automatically.
The problem is initialized from a smooth perturbation
with zero initial displacement. The initial fluid density $\rho_f(\bx,0)$
is set via the constitutive inverse $\rho_f=\rho_f(p)$
applied to
\[
\bu(\bx,0)=0,\qquad
p(\bx,0)=1+0.2\sin(\pi x)\sin(\pi y).
\]

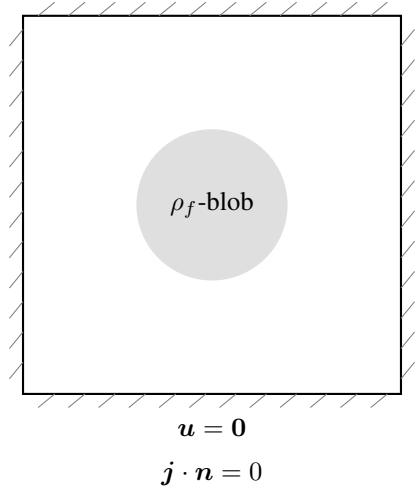
\begin{figure}[h!]
\centering
\begin{tikzpicture}[scale=1.0]
\draw[thick] (0,0) rectangle (5,5);
\foreach \x in {0.2,0.6,...,4.8} { \draw[gray] (\x,-0.18) -- (\x+0.22,0.0); }
\foreach \x in {0.2,0.6,...,4.8} { \draw[gray] (\x,5.0)   -- (\x+0.22,5.18); }
\foreach \y in {0.2,0.6,...,4.8} { \draw[gray] (-0.18,\y)  -- (0.0,\y+0.22); }
\foreach \y in {0.2,0.6,...,4.8} { \draw[gray] (5.0,\y)    -- (5.18,\y+0.22); }
\fill[gray!25] (2.5,2.5) circle (1.0);
\node at (2.5,2.5) {$\rho_f$-blob};
\node[below=6pt]  at (2.5,0) {$\bu=\boldsymbol{0}$};
\node[below=22pt] at (2.5,0) {$\boldsymbol{j}\cdot\boldsymbol{n}=0$};
\end{tikzpicture}
\caption{Schematic of the energy benchmark: unit square domain,
localized initial fluid density blob, all four walls clamped ($\bu=\boldsymbol{0}$), and
no external forcing. The total-flux no-penetration condition
$\boldsymbol{j}\cdot\boldsymbol{n}=0$ holds on all walls automatically
as the natural boundary condition of the two-field density equation.}
\label{fig:blob-schematic}
\end{figure}

Figure~\ref{fig:energy-behavior} (left) shows the normalized energy
$\mathcal{E}(t)/\mathcal{E}(0)$ on the blob relaxation problem ($K=32$,
$\delta t=10^{-3}$, $T=0.5$, benchmark parameters).
The energy decays monotonically, reaching equilibrium by $t\approx0.1$
at approximately $96.55\%$ of the initial value.
Beyond decay, the scheme's discrete energy identity should accurately
track the continuous law $\frac{d}{dt}\mathcal{E}=-\mathcal{D}$.
On the sealed blob (no external energy input), the cumulative discrepancy
\[
\mathrm{DEFECT}(T)=\Bigl|\mathcal{E}(T)-\mathcal{E}(0)+\delta t\sum_{n}D_{\mathrm{pred}}^{n}\Bigr|,
\]
where $D_{\mathrm{pred}}^{n}=\kappa\int_\Omega|\hat{\boldsymbol{d}}^{n+1/2}|^2(2/(\rho^{n+1}+\rho^n))\,dx$
is the predicted Darcy dissipation rate, converges to zero at the cumulative
$O(\delta t^2)$ rate implied by the $O(\delta t^3)$-per-step
defect~\eqref{eq:disc-2f-stab-energy}.
Note that $D_{\mathrm{pred}}^n$ is exactly the quantity appearing in the
discrete energy identity itself (see Section~\ref{sec:twofield-discretization}),
so this DEFECT test verifies~\eqref{eq:disc-2f-stab-energy} directly rather
than through a proxy.
Figure~\ref{fig:energy-behavior} (right) confirms the results with rates of $1.86$, $1.94$, $2.01$.

\pgfplotsset{
  energyplot/.style={
    width=0.47\textwidth,height=0.36\textwidth,
    grid=both,grid style={gray!20,line width=0.3pt},
    mark size=2pt,thick,
    legend style={font=\small,inner sep=2pt},
    label style={font=\small},tick label style={font=\scriptsize},
  }
}
\begin{figure}[h!]
\centering
\begin{tikzpicture}
\begin{axis}[energyplot,
  xlabel={$t$},ylabel={$\mathcal{E}(t)/\mathcal{E}(0)$},
  xmin=0,xmax=0.5,ymin=0.962,ymax=1.002,
  legend pos=north east,
]
\addplot[blue,mark=o,mark repeat=5] coordinates{
  (0.000,1.00000000)(0.010,0.98648862)(0.020,0.97881787)(0.030,0.97403835)
  (0.040,0.97099873)(0.050,0.96905240)(0.060,0.96780288)(0.070,0.96699983)
  (0.080,0.96648344)(0.090,0.96615129)(0.100,0.96593764)(0.110,0.96580019)
  (0.120,0.96571177)(0.130,0.96565489)(0.140,0.96561830)(0.150,0.96559476)
  (0.160,0.96557962)(0.170,0.96556988)(0.180,0.96556362)(0.190,0.96555959)
  (0.200,0.96555700)(0.210,0.96555533)(0.220,0.96555426)(0.230,0.96555357)
  (0.240,0.96555313)(0.250,0.96555284)(0.260,0.96555266)(0.270,0.96555254)
  (0.280,0.96555246)(0.290,0.96555241)(0.300,0.96555238)(0.310,0.96555236)
  (0.320,0.96555235)(0.330,0.96555234)(0.340,0.96555234)(0.350,0.96555233)
  (0.360,0.96555233)(0.370,0.96555233)(0.380,0.96555233)(0.390,0.96555233)
  (0.400,0.96555233)(0.410,0.96555233)(0.420,0.96555233)(0.430,0.96555233)
  (0.440,0.96555233)(0.450,0.96555233)(0.460,0.96555233)(0.470,0.96555233)
  (0.480,0.96555233)(0.490,0.96555233)(0.500,0.96555233)};
\addlegendentry{Two-field}
\end{axis}
\end{tikzpicture}
\hfill
\begin{tikzpicture}
\begin{loglogaxis}[energyplot,
  xlabel={$\delta t$},ylabel={DEFECT$(T)$},
  xmin=3e-3,xmax=0.07,ymin=1e-7,ymax=1e-4,
  legend pos=south east,
]
\addplot[blue,mark=o,thick] coordinates{
  (0.05,4.900e-5)(0.02,8.853e-6)(0.01,2.315e-6)(0.005,5.762e-7)};
\addlegendentry{Rates $1.86$, $1.94$, $2.01$}
\addplot[gray,densely dotted,thick,domain=3e-3:6e-2,samples=2]{2.0e-2*x^2};
\addlegendentry{$O(\delta t^{2})$}
\end{loglogaxis}
\end{tikzpicture}
\caption{Energy behavior of the two-field scheme on the sealed blob
  ($E=1$, $\nu=0.2$, $\kappa=1$, all boundaries clamped, no external forcing).
  \textbf{Left}: normalized energy $\mathcal{E}(t)/\mathcal{E}(0)$
  ($K=32$, $\delta t=10^{-3}$, $T=0.5$), decaying monotonically to
  ${\approx}96.55\%$ of its initial value.
  \textbf{Right}: cumulative energy-law discrepancy
  DEFECT$(T)=|\mathcal{E}(T)-\mathcal{E}(0)+\delta t\sum_n D_{\mathrm{pred}}^n|$
  ($K=64$, $T=0.5$), converging at rates $1.86$, $1.94$, $2.01\approx O(\delta t^2)$,
  confirming the $O(\delta t^3)$-per-step identity~\eqref{eq:disc-2f-stab-energy}.}
\label{fig:energy-behavior}
\end{figure}
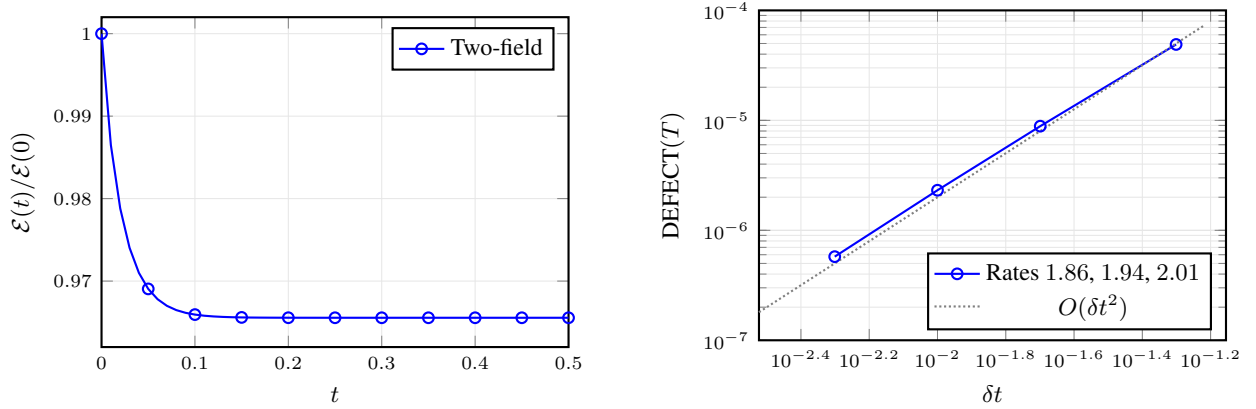

\subsection{Externally forced footing problems}\label{subsec:footing}

For a practically relevant response, we consider three externally forced footing
setups on the unit square domain. In all three cases, the
bottom boundary is clamped,
\[
\bu=\boldsymbol{0}\quad \text{on }\Gamma_b=\{y=0\},
\]
and the top boundary is subjected to downward pressure loading,
\[
\left[-\sigma_{e}(\bu)+\alpha p(\rho_f)I\right]\cdot\boldsymbol{n}=-p_{\mathrm{load}}(t)\,\mathbf{e}_y
\quad \text{on }\Gamma_t=\{y=1\},
\]
with footing parameters $E=1000$, $\nu=0.2$, $\kappa=10^{-2}$, $M=\alpha=1$.
The load follows a ramp-then-hold profile with maximum $p_{\max}=100$:
\[
p_{\mathrm{load}}(t) = \begin{cases}
p_{\max}\,t/T_{\mathrm{ramp}}, & 0\le t\le T_{\mathrm{ramp}},\\
p_{\max}, & T_{\mathrm{ramp}} < t \le T,
\end{cases}
\]
where $T_{\mathrm{ramp}}=1$ and $T=2$.
The ramp phase drives the system toward a compressed state; the hold phase
allows subsequent consolidation drainage to be observed.
Three values of the fluid energy exponent $\gamma\in\{1,2,5\}$ are compared,
corresponding to an ideal-gas ($p=M\rho_f$), quadratic, and stiff power-law
pressure--density relations. These values are chosen to span a wide range of
constitutive stiffness for methodological demonstration; they are not fit to
any specific real compressible pore fluid.  The difference between the three setups is only in the treatment of the side
boundaries, $\Gamma_L=\{x=0\}$ and $\Gamma_R=\{x=1\}$ (see Figure~\ref{fig:footing-setups}):
\begin{enumerate}[(a)]
\item \textbf{Fixed sides:} $\bu=\boldsymbol{0}$ on $\Gamma_L\cup\Gamma_R$.
\item \textbf{Roller sides:} tangentially free but normal displacement fixed,
			i.e. $\bu\cdot\boldsymbol{n}=0$ on $\Gamma_L\cup\Gamma_R$ and
			$(I-\boldsymbol{n}\otimes\boldsymbol{n})\sigma\boldsymbol{n}=\boldsymbol{0}$.
\item \textbf{Free sides:} traction free,
			$\sigma(\bu,\rho_f)\boldsymbol{n}=\boldsymbol{0}$ on $\Gamma_L\cup\Gamma_R$.
\end{enumerate}

The total no-flux condition $\boldsymbol{j}\cdot\boldsymbol{n}=0$ arises as the
natural boundary condition of the density residual and is automatically satisfied
on all sealed boundaries.
For the footing experiments we apply:
\begin{itemize}
\item \textbf{Bottom and sides}: sealed (natural no-flux condition).
\item \textbf{Top} ($\Gamma_t$, loaded surface): drained,
  $\rho_{f}=\rho_{\mathrm{ref}}$ as a Dirichlet condition
  (no excess pore pressure at the drainage surface; standard consolidation BC).
\end{itemize}

\begin{figure}[h!]
\centering
\begin{minipage}[t]{0.32\textwidth}
\centering
\begin{tikzpicture}[scale=1.0]
\draw[thick] (0,0) rectangle (2.2,2.2);
\draw[->,thick,blue!70!black] (0.4,2.55) -- (0.4,2.25);
\draw[->,thick,blue!70!black] (1.1,2.55) -- (1.1,2.25);
\draw[->,thick,blue!70!black] (1.8,2.55) -- (1.8,2.25);
\foreach \x in {0.1,0.3,...,2.1} {
	\draw[gray] (\x,-0.14) -- (\x+0.16,0.0);
}
\draw[red!70!black,thick] (-0.18,0) -- (-0.18,2.2);
\draw[red!70!black,thick] (2.38,0) -- (2.38,2.2);
\node at (1.1,-0.35) {clamped bottom};
\node[blue!70!black] at (1.1,2.78) {$p_{\mathrm{load}}\downarrow$};
\node[red!70!black,rotate=90] at (-0.35,1.1) {fixed};
\node[red!70!black,rotate=90] at (2.55,1.1) {fixed};
\end{tikzpicture}

\small (a) Fixed side boundaries
\end{minipage}\hfill
\begin{minipage}[t]{0.32\textwidth}
\centering
\begin{tikzpicture}[scale=1.0]
\draw[thick] (0,0) rectangle (2.2,2.2);
\draw[->,thick,blue!70!black] (0.4,2.55) -- (0.4,2.25);
\draw[->,thick,blue!70!black] (1.1,2.55) -- (1.1,2.25);
\draw[->,thick,blue!70!black] (1.8,2.55) -- (1.8,2.25);
\foreach \x in {0.1,0.3,...,2.1} {
	\draw[gray] (\x,-0.14) -- (\x+0.16,0.0);
}
\draw[red!70!black,thick] (-0.18,0) -- (-0.18,2.2);
\draw[red!70!black,thick] (2.38,0) -- (2.38,2.2);
\draw[->,orange!80!black,thick] (-0.58,1.1) -- (-0.27,1.1);
\draw[->,orange!80!black,thick] (2.78,1.1) -- (2.47,1.1);
\draw[<->,teal!70!black,thick] (0.08,0.45) -- (0.08,1.75);
\draw[<->,teal!70!black,thick] (2.12,0.45) -- (2.12,1.75);
\node at (1.1,-0.35) {clamped bottom};
\node[blue!70!black] at (1.1,2.78) {$p_{\mathrm{load}}\downarrow$};
\node[red!70!black,rotate=90] at (-0.78,1.1) {$u_n=0$};
\node[red!70!black,rotate=90] at (2.98,1.1) {$u_n=0$};
\node[orange!80!black,fill=white,inner sep=1pt] at (1.1,1.42) {normal lock};
\node[teal!70!black,fill=white,inner sep=1pt] at (1.1,0.86) {tangentially free};
\end{tikzpicture}

\small (b) Roller side boundaries
\end{minipage}\hfill
\begin{minipage}[t]{0.32\textwidth}
\centering
\begin{tikzpicture}[scale=1.0]
\draw[thick] (0,0) rectangle (2.2,2.2);
\draw[->,thick,blue!70!black] (0.4,2.55) -- (0.4,2.25);
\draw[->,thick,blue!70!black] (1.1,2.55) -- (1.1,2.25);
\draw[->,thick,blue!70!black] (1.8,2.55) -- (1.8,2.25);
\foreach \x in {0.1,0.3,...,2.1} {
	\draw[gray] (\x,-0.14) -- (\x+0.16,0.0);
}
\draw[green!60!black,thick,dashed] (-0.18,0) -- (-0.18,2.2);
\draw[green!60!black,thick,dashed] (2.38,0) -- (2.38,2.2);
\node at (1.1,-0.35) {clamped bottom};
\node[blue!70!black] at (1.1,2.78) {$p_{\mathrm{load}}\downarrow$};
\node[green!60!black,rotate=90] at (-0.35,1.1) {free};
\node[green!60!black,rotate=90] at (2.55,1.1) {free};
\end{tikzpicture}

\small (c) Free side boundaries
\end{minipage}
\caption{Footing configurations on $(0,1)^2$ with common bottom clamp
and top downward pressure loading. Side-boundary treatment varies across
the three setups: fixed, roller (normal lock with tangential freedom), and free.}
\label{fig:footing-setups}
\end{figure}
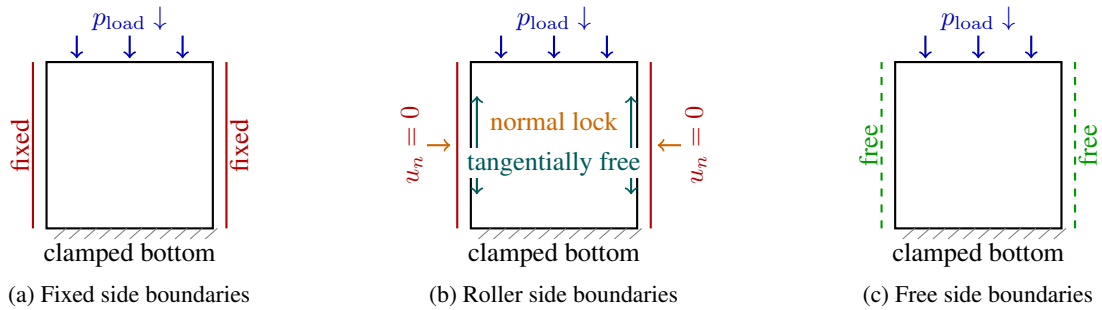

In all three setups, we study the two-field formulation in terms of
displacement response, pore-pressure/density evolution, drainage pattern,
and robustness of nonlinear solves under increasing load intensity.  The comparison focuses on how lateral confinement changes compaction, lateral
bulging, pore-pressure concentration near the loaded boundary, and
post-load consolidation drainage.

The two-field runs completed for all nine configurations ($K=128$, $\delta t=10^{-3}$,
$T=2$, $p_{\max}=100$, $T_{\mathrm{ramp}}=1$) in $1.5$--$1.9$ Newton iterations per
step on average, consistent with the mild bilinear nonlinearity noted in
Remark~\ref{rem:coupling-nonlinearity}.

Figure~\ref{fig:footing-histories} shows the stored energy
$\mathcal{E}(t)-\mathcal{E}(0)$.  During the ramp ($t\in[0,1]$), energy accumulates nonlinearly.
Fixed sides store the least energy (highest lateral confinement);
roller and free sides store progressively more.
Stored energy increases with $\gamma$ (stiffer fluid).
During the hold ($t\in[1,2]$), the system consolidates: energy decreases as elastic
stress relaxes through drainage.
The consolidation is most pronounced for $\gamma=5$, where the stiffer fluid
sustains larger pore pressures that drive stronger Darcy flow; for $\gamma=1$ the
hold-phase change is modest.

\pgfplotsset{
  footplot2/.style={
    width=0.32\textwidth,height=0.27\textwidth,
    grid=both,grid style={gray!20,line width=0.3pt},
    mark size=0pt,thick,
    xlabel={$t$},xmin=0,xmax=2,
    label style={font=\small},tick label style={font=\scriptsize},
    title style={font=\small},
    legend style={font=\scriptsize,inner sep=2pt,row sep=-2pt},
  }
}
\begin{figure}[t]
\centering
\begin{tikzpicture}
\begin{axis}[footplot2,title={(a) $\gamma=1$: stored energy},
  ylabel={$\mathcal{E}-\mathcal{E}_0$},ymin=0,ymax=5.5,legend pos=north west]
\addplot[blue,solid] coordinates{(0,0)(0.2,0.0936)(0.4,0.3579)(0.6,0.7930)(0.8,1.3989)(1.0,2.1756)(1.2,2.1739)(1.4,2.1727)(1.6,2.1717)(1.8,2.1709)(2.0,2.1702)};
\addlegendentry{fixed}
\addplot[red,dashed] coordinates{(0,0)(0.2,0.1976)(0.4,0.7549)(0.6,1.6719)(0.8,2.9485)(1.0,4.5848)(1.2,4.5823)(1.4,4.5804)(1.6,4.5789)(1.8,4.5777)(2.0,4.5766)};
\addlegendentry{roller}
\addplot[green!60!black,dotted,ultra thick] coordinates{(0,0)(0.2,0.2088)(0.4,0.7978)(0.6,1.7671)(0.8,3.1166)(1.0,4.8464)(1.2,4.8446)(1.4,4.8428)(1.6,4.8414)(1.8,4.8401)(2.0,4.8391)};
\addlegendentry{free}
\draw[dashed,gray](axis cs:1,0)--(axis cs:1,5.5);
\end{axis}
\end{tikzpicture}
\hfill
\begin{tikzpicture}
\begin{axis}[footplot2,title={(b) $\gamma=2$: stored energy},
  ylabel={$\mathcal{E}-\mathcal{E}_0$},ymin=0,ymax=5.5,legend pos=north west]
\addplot[blue,solid] coordinates{(0,0)(0.2,0.1011)(0.4,0.3722)(0.6,0.8137)(0.8,1.4255)(1.0,2.2078)(1.2,2.2043)(1.4,2.2017)(1.6,2.1996)(1.8,2.1978)(2.0,2.1962)};
\addlegendentry{fixed}
\addplot[red,dashed] coordinates{(0,0)(0.2,0.2144)(0.4,0.7873)(0.6,1.7191)(0.8,3.0098)(1.0,4.6594)(1.2,4.6544)(1.4,4.6504)(1.6,4.6470)(1.8,4.6441)(2.0,4.6414)};
\addlegendentry{roller}
\addplot[green!60!black,dotted,ultra thick] coordinates{(0,0)(0.2,0.2264)(0.4,0.8319)(0.6,1.8166)(0.8,3.1808)(1.0,4.9244)(1.2,4.9203)(1.4,4.9164)(1.6,4.9130)(1.8,4.9100)(2.0,4.9072)};
\addlegendentry{free}
\draw[dashed,gray](axis cs:1,0)--(axis cs:1,5.5);
\end{axis}
\end{tikzpicture}
\hfill
\begin{tikzpicture}
\begin{axis}[footplot2,title={(c) $\gamma=5$: stored energy},
  ylabel={$\mathcal{E}-\mathcal{E}_0$},ymin=0,ymax=5.5,legend pos=north west]
\addplot[blue,solid] coordinates{(0,0)(0.2,0.1170)(0.4,0.3977)(0.6,0.8452)(0.8,1.4605)(1.0,2.2440)(1.2,2.2294)(1.4,2.2187)(1.6,2.2099)(1.8,2.2023)(2.0,2.1956)};
\addlegendentry{fixed}
\addplot[red,dashed] coordinates{(0,0)(0.2,0.2526)(0.4,0.8487)(0.6,1.7914)(0.8,3.0817)(1.0,4.7203)(1.2,4.6906)(1.4,4.6671)(1.6,4.6475)(1.8,4.6307)(2.0,4.6162)};
\addlegendentry{roller}
\addplot[green!60!black,dotted,ultra thick] coordinates{(0,0)(0.2,0.2662)(0.4,0.8950)(0.6,1.8900)(0.8,3.2521)(1.0,4.9823)(1.2,4.9541)(1.4,4.9300)(1.6,4.9097)(1.8,4.8924)(2.0,4.8775)};
\addlegendentry{free}
\draw[dashed,gray](axis cs:1,0)--(axis cs:1,5.5);
\end{axis}
\end{tikzpicture}

\caption{Two-field footing stored energy ($K=128$, $\delta t=10^{-3}$, $T=2$,
  $E=1000$, $\nu=0.2$, $\kappa=10^{-2}$, $p_{\max}=100$, $T_{\mathrm{ramp}}=1$)
  for $\gamma=1,2,5$;
  fixed (blue solid), roller (red dashed), free (green dotted) side boundaries.
  The gray vertical line marks $T_{\mathrm{ramp}}=1$.
  During the hold phase ($t>1$) the system consolidates: energy dissipates,
  most strongly for $\gamma=5$ (stiffer fluid, larger pore pressures).}
\label{fig:footing-histories}
\end{figure}
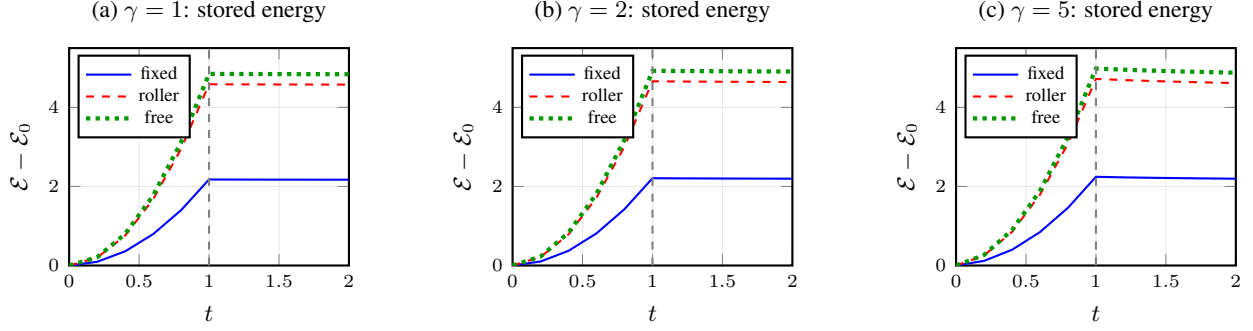

\begin{figure}[tb]
\centering
\includegraphics[width=\textwidth]{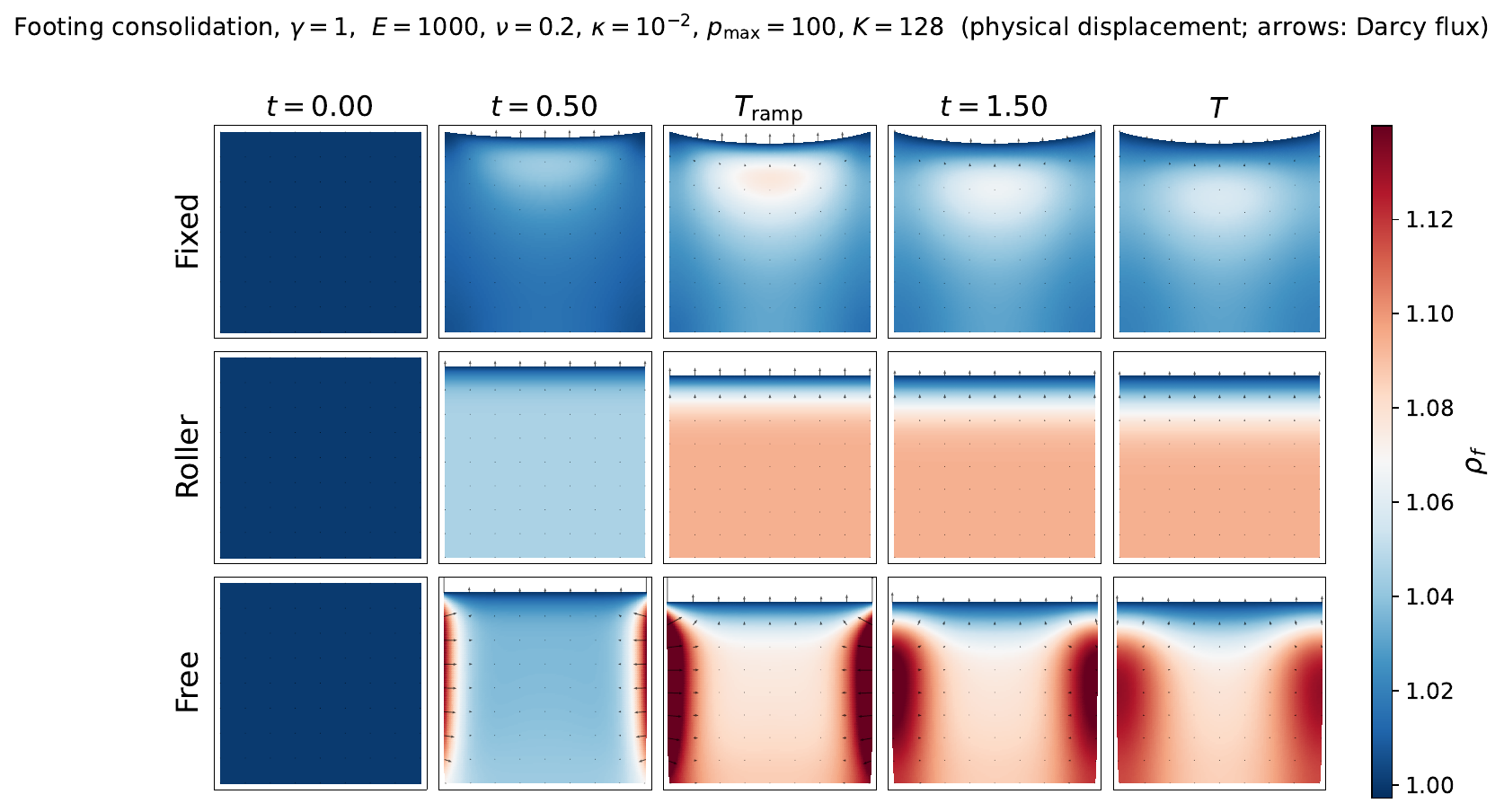}
\caption{Deformed-domain evolution for $\gamma=1$ ($p=M\rho_f$, ideal gas).
  Layout: rows are fixed (top), roller (middle), and free (bottom);
  columns are five equally spaced times $t\in\{0,\,0.5,\,\ldots,\,2\}$,
  with $t=1$ the end of the ramp and $t=2$ the end of the hold.
  Mesh deformed by the physical displacement field $\bu$ (no amplification);
  color is the fluid density $\rho_f$ on a shared scale.
  Parameters: $E=1000$, $\nu=0.2$, $\kappa=10^{-2}$, $M=\alpha=1$,
  $p_{\max}=100$, $T_{\mathrm{ramp}}=1$, $K=128$, $\delta t=10^{-3}$.}
\label{fig:footing-deformed-g1}
\end{figure}

\begin{figure}[tb]
\centering
\includegraphics[width=\textwidth]{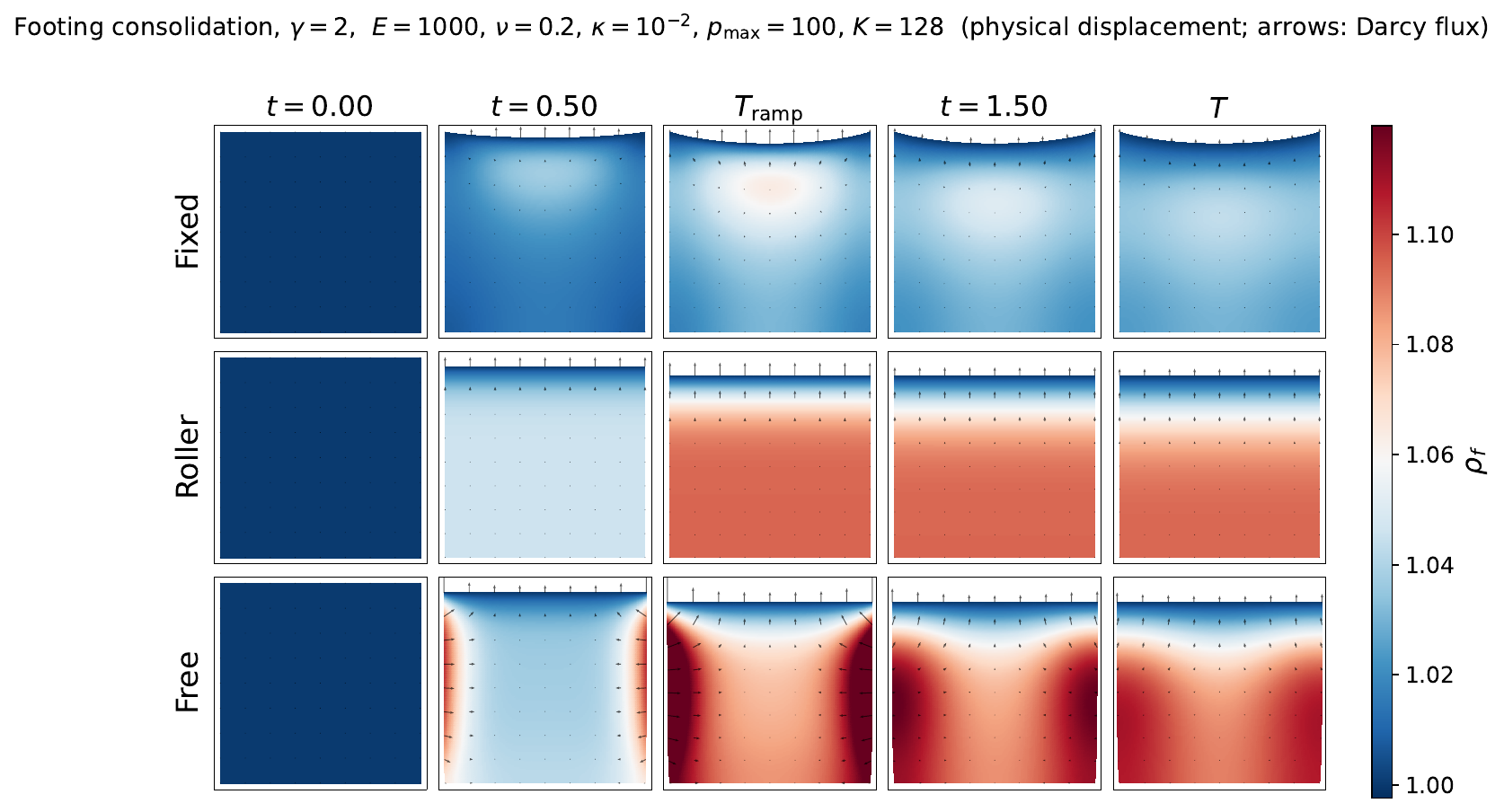}
\caption{Deformed-domain evolution for $\gamma=2$
  ($p=M(\gamma-1)\rho_f^2$, quadratic law).
  Layout and parameters as in Fig.~\ref{fig:footing-deformed-g1}.}
\label{fig:footing-deformed-g2}
\end{figure}

\begin{figure}[tb]
\centering
\includegraphics[width=\textwidth]{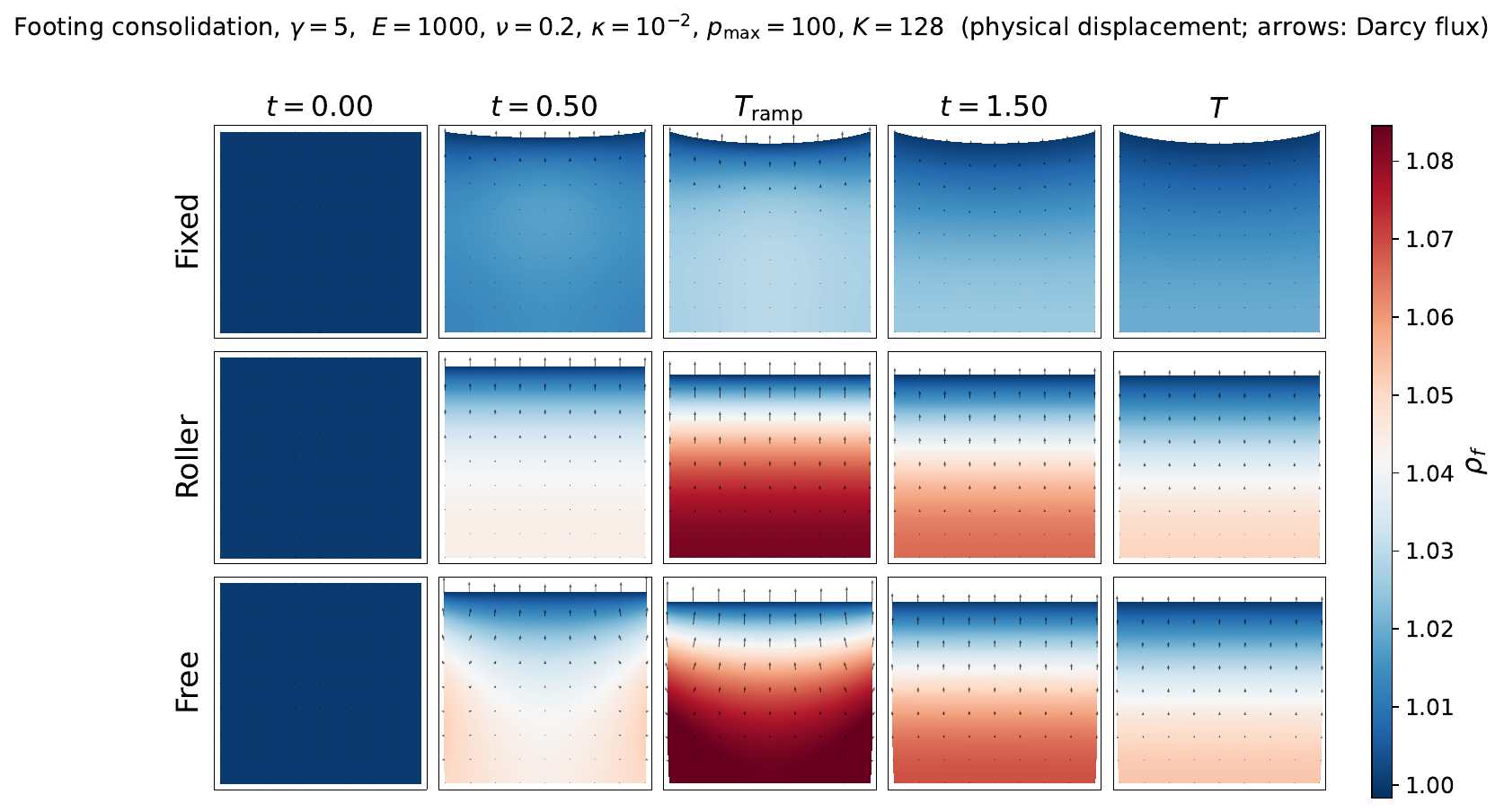}
\caption{Deformed-domain evolution for $\gamma=5$
  ($p=M(\gamma-1)\rho_f^5$, stiff power law).
  Layout and parameters as in Fig.~\ref{fig:footing-deformed-g1}.}
\label{fig:footing-deformed-g5}
\end{figure}

Figures~\ref{fig:footing-deformed-g1}--\ref{fig:footing-deformed-g5} show the
deformed domain, colored by fluid density $\rho_f$, at five equally spaced
times. Displacement is shown at physical scale (no amplification); at
$p_{\max}=100$ it reaches $5$--$10\%$ of the domain size, large enough that
the infinitesimal-strain elasticity used throughout is only an approximation,
adopted here for consistency with the manufactured-solution and energy tests
of Sections~\ref{subsec:mfld-accuracy} and~\ref{subsec:energy-2f} rather than
to model finite-strain behavior. Several features stand out.

\emph{Density range decreases with $\gamma$.}
For $\gamma=1$ the peak excess density reaches $\rho_f-1\approx 23\%$ (free
sides at $t=T_{\mathrm{ramp}}$), decreasing to $\approx16\%$ for $\gamma=2$
and to $\approx9\%$ for $\gamma=5$.
This inverse trend with $\gamma$ reflects the constitutive stiffness
\[
\left.\frac{dp}{d\rho_f}\right|_{\rho_f=1} =
\begin{cases}
M, & \gamma=1\ \text{(ideal gas, }\omega=M\rho\ln\rho,\ p=M\rho_f\text{)},\\[2pt]
M(\gamma-1)\gamma\rho_f^{\gamma-1}, & \gamma>1\ \text{(power law, }\omega=M\rho^\gamma\text{)},
\end{cases}
\]
which equals $1$, $2$, and $20$ at $\rho_f=1$ for $\gamma=1,2,5$ respectively.
A stiffer pressure--density relation requires a larger pore-pressure increment
to produce the same density change, so the fluid density responds less to the
applied load as $\gamma$ increases.

\emph{Lateral confinement concentrates fluid beneath the footing.}
Free sides allow the skeleton to bulge laterally, which compresses the fluid
more uniformly over the domain and drives larger density concentrations
immediately below the loaded surface.
Fixed sides suppress lateral motion, distributing the compression more
uniformly and producing the smallest density peaks.
Roller sides fall between the two: they allow vertical sliding but lock
normal displacement, leading to a density pattern intermediate in magnitude
but more laterally uniform than the free case.

\emph{Hold-phase drainage is controlled by the effective drainage timescale.}
During the hold ($t\in[1,2]$), excess pore pressure relaxes by drainage
through the top boundary.
The characteristic timescale is $\tau = L^2/(\kappa\,dp/d\rho_f)$, where $L=1$
is the domain side length, which evaluates to $\tau=100$, $50$, and $5$ for
$\gamma=1,2,5$ respectively.
Since $T_{\mathrm{hold}}=1$, only a fraction $\sim T_{\mathrm{hold}}/\tau$
of the excess pore pressure can drain during the hold phase.
For $\gamma=1$ and $\gamma=2$ (roller case), the density map is essentially
frozen at its ramp-end value: the data confirm that $\max\rho_f$ stays at
$1.009$ through all hold-phase snapshots for both exponents.
For $\gamma=5$, with $\tau\approx T_{\mathrm{hold}}$, the density visibly
relaxes, decreasing from its peak by roughly $30$--$40\%$ of the excess
by $t=2$, which is consistent with the stronger hold-phase consolidation seen in
Figs.~\ref{fig:footing-histories} and the summary given in Table~\ref{tab:footing-summary}.

\begin{table}[h!]
\centering
\caption{Two-field footing summary ($p_{\max}=100$, $K=128$, $\delta t=10^{-3}$).
  Stored energy $\Delta\mathcal{E}=\mathcal{E}-\mathcal{E}_0$
  and fluid mass change $\Delta\mathcal{M}=\mathcal{M}-\mathcal{M}_0$ (units $10^{-2}$) at end of ramp ($T_{\mathrm{ramp}}=1$)
  and end of hold ($T=2$), showing consolidation during the hold phase.}
\label{tab:footing-summary}
\smallskip
\begin{tabular}{llcccc}
\toprule
& & \multicolumn{2}{c}{$\Delta\mathcal{E}$} & \multicolumn{2}{c}{$\Delta\mathcal{M}$, $\times10^{-2}$} \\
\cmidrule(lr){3-4}\cmidrule(lr){5-6}
$\gamma$ & Side & $t=1$ & $t=2$ & $t=1$ & $t=2$ \\
\midrule
1 & Fixed  & 2.18 & 2.17 & 3.76 & 3.36 \\
1 & Roller & 4.58 & 4.58 & 8.48 & 7.99 \\
1 & Free   & 4.85 & 4.84 & 8.97 & 8.45 \\
\midrule
2 & Fixed  & 2.21 & 2.20 & 3.54 & 3.01 \\
2 & Roller & 4.66 & 4.64 & 8.19 & 7.43 \\
2 & Free   & 4.92 & 4.91 & 8.65 & 7.85 \\
\midrule
5 & Fixed  & 2.24 & 2.20 & 2.35 & 1.31 \\
5 & Roller & 4.72 & 4.62 & 5.88 & 3.32 \\
5 & Free   & 4.98 & 4.88 & 6.18 & 3.49 \\
\bottomrule
\end{tabular}
\end{table}

\section{Conclusions}
\label{sec:conclusions}

In this paper, we have derived a nonlinear poroelastic model from an energy-dissipation
variational principle. The constitutive pressure relation
$p(\rho_f)=\rho_f\omega_\rho-\omega$ emerges directly from variation;
the total-flux transport structure follows from the kinematic coupling
assumption. Thermodynamic consistency is guaranteed by construction, and
extensions to thermal effects, reactive transport, or multi-component fluids
require only augmenting the free-energy and dissipation functionals.

A compatible two-field $(\bu,\rho_f)$ discretization inherits a discrete
energy identity with an $O(\delta t^3)$-per-step defect.
Manufactured-solution tests under genuine displacement--density coupling
confirm second-order convergence in both time and space for both fields,
measured in the energy quantities the model itself controls.
Footing experiments with three lateral-boundary conditions and three
compressibility exponents $\gamma\in\{1,2,5\}$ show that the effective
drainage timescale $\tau=L^2/(\kappa\,p'(\rho_f))$ controls hold-phase
consolidation: negligible for $\gamma=1$, clearly visible for $\gamma=5$.  The two-field variational scheme gives provable $O(\delta t^2)$ energy-law accuracy,
requires only a mildly nonlinear solve (i.e., one or two Newton iterations per time
step in practice, see Remark~\ref{rem:coupling-nonlinearity}), and performs robustly
on physically demanding footing problems in the under-drained regime.

Future directions include a dedicated numerical-analysis study of both the
two-field and three-field discretizations.  This will cover accuracy and
energy-law verification across the full range of fluid-compressibility
exponents, mesh and time-step sensitivity for applications such as the
footing problem, and preconditioner performance.  Additionally, we will study coupled
thermal and reactive extensions via the same variational framework, and
three-dimensional geomechanical applications.

\section*{Code and data availability}
The two-field finite-element solver, the experiment drivers behind every figure and
table in this paper, and the corresponding data are available at
\url{https://github.com/akirshtein/variational-poroelasticity-model}.

\section*{Acknowledgments}
This work was partially supported by the National Science Foundation (NSF) under grant DMS-2208267.

\bibliographystyle{unsrtnat}
\bibliography{references}

@article{pride1992deriving,
	author  = {Pride, Steven R. and Gangi, Anthony F. and Morgan, F. Dale},
	title   = {Deriving the Equations of Motion for Porous Isotropic Media},
	journal = {Journal of the Acoustical Society of America},
	year    = {1992},
	volume  = {92},
	number  = {6},
	pages   = {3278--3290},
	doi     = {10.1121/1.404178}
}

@article{biot1962mechanics,
	author  = {Biot, Maurice A.},
	title   = {Mechanics of Deformation and Acoustic Propagation in Porous Media},
	journal = {Journal of Applied Physics},
	year    = {1962},
	volume  = {33},
	number  = {4},
	pages   = {1482--1498},
	doi     = {10.1063/1.1728759}
}

@article{bowen1980incompressible,
	author  = {Bowen, Ray M.},
	title   = {Incompressible Porous Media Models by Use of the Theory of Mixtures},
	journal = {International Journal of Engineering Science},
	year    = {1980},
	volume  = {18},
	number  = {9},
	pages   = {1129--1148},
	doi     = {10.1016/0020-7225(80)90114-7}
}

@article{burridge1981poroelasticity,
	author  = {Burridge, Robert and Keller, Joseph B.},
	title   = {Poroelasticity Equations Derived from Microstructure},
	journal = {Journal of the Acoustical Society of America},
	year    = {1981},
	volume  = {70},
	number  = {4},
	pages   = {1140--1146},
	doi     = {10.1121/1.386945}
}

@book{terzaghi1943theoretical,
	author    = {Terzaghi, Karl},
	title     = {Theoretical Soil Mechanics},
	publisher = {John Wiley \& Sons},
	address   = {New York, NY},
	year      = {1943}
}

@article{steeb2019mechanics,
	author  = {Steeb, Holger and Renner, J{\"o}rg},
	title   = {Mechanics of Poro-Elastic Media: A Review with Emphasis on
	Foundational State Variables},
	journal = {Transport in Porous Media},
	year    = {2019},
	volume  = {130},
	number  = {2},
	pages   = {437--461},
	doi     = {10.1007/s11242-019-01319-6}
}

@article{coussy1998frommixture,
	author  = {Coussy, Olivier and Dormieux, Luc and Detournay, Emmanuel},
	title   = {From Mixture Theory to {B}iot's Approach for Porous Media},
	journal = {International Journal of Solids and Structures},
	year    = {1998},
	volume  = {35},
	number  = {34--35},
	pages   = {4619--4634},
	doi     = {10.1016/S0020-7683(98)00070-6}
}

@book{auriault2009homogenization,
	author    = {Auriault, Jean-Louis and Boutin, Claude and Geindreau, Christian},
	title     = {Homogenization of Coupled Phenomena in Heterogeneous Media},
	publisher = {ISTE / John Wiley \& Sons},
	address   = {London, UK},
	year      = {2009}
}

@book{giga2018variational,
	publisher={Springer},
	series={Springer Books},
	edition={None},
	booktitle={Handbook of Mathematical Analysis in Mechanics of Viscous Fluids},
	chapter={2},
	author={Mi-Ho Giga and Arkadz Kirshtein and Chun Liu},
	title={Variational Modeling and Complex Fluids},
	year={2018},
	month={March},
	pages={73-113},
	volume={None},
	doi={10.1007/978-3-319-13344-7_2},
}

@book{marsden1978foundations,
  title={Foundations of Mechanics},
  author={Abraham, Ralph and Marsden, Jerrold E.},
  edition={2nd},
  year={1978},
  publisher={Addison-Wesley}
}

@book{gurtin1981continuum,
  title={An Introduction to Continuum Mechanics},
  author={Gurtin, Morton E.},
  year={1981},
  publisher={Academic Press}
}

@article{biot1941general,
  title={General Theory of Three-Dimensional Consolidation},
  author={Biot, Maurice A.},
  journal={Journal of Applied Physics},
  volume={12},
  number={2},
  pages={155--164},
  year={1941},
  doi={10.1063/1.1712886},
  url={https://doi.org/10.1063/1.1712886}
}

@article{biot1955theory,
  title={Theory of Elasticity and Consolidation for a Porous Anisotropic Solid},
  author={Biot, Maurice A.},
  journal={Journal of Applied Physics},
  volume={26},
  number={2},
  pages={182--185},
  year={1955},
  doi={10.1063/1.1721956},
  url={https://doi.org/10.1063/1.1721956}
}

@book{coussy2004poromechanics,
  title={Poromechanics},
  author={Coussy, Olivier},
  year={2004},
  publisher={John Wiley \& Sons},
  isbn={9780470849201},
  url={https://onlinelibrary.wiley.com/doi/book/10.1002/0470092718}
}

@book{lewis1998finite,
  title={The Finite Element Method in the Static and Dynamic Deformation and Consolidation of Porous Media},
  author={Lewis, Roland W. and Schrefler, Bernard A.},
  edition={2nd},
  year={1998},
  publisher={John Wiley \& Sons},
  isbn={9780471978221}
}

@article{deanna2019nonisothermal,
  title={Non-isothermal General {Ericksen--Leslie} System: Derivation, Analysis and Thermodynamic Consistency},
  author={De Anna, Francesco and Liu, Chun},
  journal={Archive for Rational Mechanics and Analysis},
  volume={231},
  number={2},
  pages={637--717},
  year={2019},
  doi={10.1007/s00205-018-1287-4}
}

@article{liu2018nonisothermal,
  title={Non-isothermal electrokinetics: energetic variational approach},
  author={Liu, Pei and Wu, Simo and Liu, Chun},
  journal={Communications in Mathematical Sciences},
  volume={16},
  number={5},
  pages={1451--1463},
  year={2018},
  doi={10.4310/CMS.2018.v16.n5.a13}
}

@article{wang2020field,
  title={Field theory of reaction-diffusion: {Law} of mass action with an energetic variational approach},
  author={Wang, Yiwei and Liu, Chun and Liu, Pei and Eisenberg, Bob},
  journal={Physical Review E},
  volume={102},
  number={6},
  pages={062147},
  year={2020},
  doi={10.1103/PhysRevE.102.062147}
}

@incollection{brannick2016dynamics,
  title={Dynamics of Multi-Component Flows: Diffusive Interface Methods With Energetic Variational Approaches},
  author={Brannick, James and Kirshtein, Arkadz and Liu, Chun},
  booktitle={Reference Module in Materials Science and Materials Engineering},
  publisher={Elsevier},
  year={2016},
  doi={10.1016/B978-0-12-803581-8.03624-9}
}

@article{chapelle2014,
  title={General coupling of porous flows and hyperelastic formulations---{F}rom thermodynamics principles to energy balance and compatible time schemes},
  author={Chapelle, Dominique and Moireau, Philippe},
  journal={European Journal of Mechanics - B/Fluids},
  volume={46},
  pages={82--96},
  year={2014},
  doi={10.1016/j.euromechflu.2014.02.009}
}

@article{rohan2017,
  title={Modelling large-deforming fluid-saturated porous media using an {E}ulerian incremental formulation},
  author={Rohan, Eduard and Luke{\v{s}}, Vladim{\'i}r},
  journal={Advances in Engineering Software},
  volume={113},
  pages={84--95},
  year={2017},
  doi={10.1016/j.advengsoft.2016.11.003}
}

@misc{baratta2023dolfinx,
  title={{DOLFINx}: The next generation {FEniCS} problem solving environment},
  author={Baratta, Igor A. and Dean, Joseph P. and Dokken, J{\o}rgen S. and Habera, Michal and Hale, Jack S. and Richardson, Chris N. and Rognes, Marie E. and Scroggs, Matthew W. and Sime, Nathan and Wells, Garth N.},
  year={2023},
  publisher={Zenodo},
  doi={10.5281/zenodo.10447666}
}

@inproceedings{balay1997petsc,
  author={Balay, Satish and Gropp, William D. and McInnes, Lois Curfman and Smith, Barry F.},
  title={Efficient Management of Parallelism in Object Oriented Numerical Software Libraries},
  booktitle={Modern Software Tools in Scientific Computing},
  editor={Arge, E. and Bruaset, A. M. and Langtangen, H. P.},
  publisher={Birkh{\"a}user Press},
  pages={163--202},
  year={1997}
}

@incollection{falgout2002hypre,
  author={Falgout, Robert D. and Yang, Ulrike Meier},
  title={{hypre}: A Library of High Performance Preconditioners},
  booktitle={Computational Science --- {ICCS} 2002},
  series={Lecture Notes in Computer Science},
  volume={2331},
  pages={632--641},
  publisher={Springer},
  address={Berlin, Heidelberg},
  year={2002},
  doi={10.1007/3-540-47789-6_66}
}

@article{henson2002boomeramg,
  author={Henson, Van Emden and Yang, Ulrike Meier},
  title={{BoomerAMG}: A Parallel Algebraic Multigrid Solver and Preconditioner},
  journal={Applied Numerical Mathematics},
  volume={41},
  number={1},
  pages={155--177},
  year={2002},
  doi={10.1016/S0168-9274(01)00115-5}
}

@book{saad2003iterative,
  author={Saad, Yousef},
  title={Iterative Methods for Sparse Linear Systems},
  edition={2nd},
  publisher={SIAM},
  address={Philadelphia, PA},
  year={2003},
  isbn={0-89871-534-2}
}

@article{saad1986gmres,
  author={Saad, Youcef and Schultz, Martin H.},
  title={{GMRES}: A Generalized Minimal Residual Algorithm for Solving Nonsymmetric Linear Systems},
  journal={SIAM Journal on Scientific and Statistical Computing},
  volume={7},
  number={3},
  pages={856--869},
  year={1986},
  doi={10.1137/0907058}
}

@article{bregman1967relaxation,
  author={Bregman, Lev M.},
  title={The Relaxation Method of Finding the Common Point of Convex Sets and Its Application to the Solution of Problems in Convex Programming},
  journal={USSR Computational Mathematics and Mathematical Physics},
  volume={7},
  number={3},
  pages={200--217},
  year={1967},
  doi={10.1016/0041-5553(67)90040-7}
}

@book{brenner2008mathematical,
  author={Brenner, Susanne C. and Scott, L. Ridgway},
  title={The Mathematical Theory of Finite Element Methods},
  edition={3rd},
  series={Texts in Applied Mathematics},
  volume={15},
  publisher={Springer},
  address={New York, NY},
  year={2008},
  isbn={978-0-387-75933-3},
  doi={10.1007/978-0-387-75934-0}
}

@book{kelley1995iterative,
  author={Kelley, C. T.},
  title={Iterative Methods for Linear and Nonlinear Equations},
  series={Frontiers in Applied Mathematics},
  volume={16},
  publisher={SIAM},
  address={Philadelphia, PA},
  year={1995},
  isbn={0-89871-352-8}
}

@book{evans2010partial,
  author={Evans, Lawrence C.},
  title={Partial Differential Equations},
  edition={2nd},
  series={Graduate Studies in Mathematics},
  volume={19},
  publisher={American Mathematical Society},
  address={Providence, RI},
  year={2010},
  isbn={978-0-8218-4974-3}
}

@book{ciarlet1988elasticity,
  author={Ciarlet, Philippe G.},
  title={Mathematical Elasticity. Volume {I}: Three-Dimensional Elasticity},
  series={Studies in Mathematics and its Applications},
  volume={20},
  publisher={North-Holland},
  address={Amsterdam},
  year={1988},
  isbn={0-444-70259-8}
}

@book{jungel2016entropy,
  author={J{\"u}ngel, Ansgar},
  title={Entropy Methods for Diffusive Partial Differential Equations},
  series={SpringerBriefs in Mathematics},
  publisher={Springer},
  address={Cham},
  year={2016},
  doi={10.1007/978-3-319-34219-1}
}

\end{document}